\pdfoutput=1
\documentclass[a4paper,11pt,english]{texbase/shinyart}
\usepackage[utf8]{inputenc}
\usepackage{texbase/shinybib}

\usepackage{todonotes}

\usepackage{rotating}

\usepackage{autonum}

\usepackage{lipsum}
\usepackage{amsfonts}
\usepackage{graphicx}
\usepackage{epstopdf}
\usepackage{subcaption}
\usepackage{algorithmic}

\usepackage{amsmath}
\usepackage{amssymb}
\usepackage{mathrsfs}
\usepackage{float}
\floatstyle{ruled}
\newfloat{algorithm}{tpb!}{loa}
\floatname{algorithm}{Algorithm}
\crefname{algorithm}{Algorithm}{Algorithms}

\DeclareMathOperator*{\argmin}{argmin}

\ifpdf
  \DeclareGraphicsExtensions{.eps,.pdf,.png,.jpg}
\else
  \DeclareGraphicsExtensions{.eps}
\fi

\bibliography{references}

\newcommand{\TheTitle}{Spherical function regularization for parallel MRI reconstruction}
\newcommand{\TheAuthors}{Y. Zhu and T. Valkonen}

\title{\TheTitle}

\author{
  Yonggui Zhu\thanks{School of Science, Communication University of China, P. R. China
    (\email{ygzhu@cuc.edu.cn})},
  \and
  Tuomo Valkonen\thanks{Department of Mathematical Sciences, University of Liverpool, United Kingdom (\email{tuomo.valkonen@iki.fi}).}
}

\usepackage{amsopn}
\DeclareMathOperator{\diag}{diag}

\ifpdf
\hypersetup{
  pdftitle={\TheTitle},
  pdfauthor={\TheAuthors}
}
\fi

\begin{document}

\maketitle

\begin{abstract}
    From the optimization point of view, a difficulty with parallel MRI with simultaneous coil sensitivity estimation is the multiplicative nature of the non-linear forward operator: the image being reconstructed and the coil sensitivities compete against each other, causing the optimization process to be very sensitive to small perturbations.
    This can, to some extent, be avoided by regularizing the unknown in a suitably ``orthogonal'' fashion. In this paper, we introduce such a regularization based on spherical function bases. To perform this regularization, we represent efficient recurrence formulas for spherical Bessel functions and associated Legendre functions. Numerically, we study the solution of the model with non-linear ADMM. We perform various numerical simulations to demonstrate the efficacy of the proposed model in parallel MRI reconstruction.

\end{abstract}

\section{Introduction}

\subsection{Parallel magnetic resonance imaging}

Parallel magnetic resonance imaging (p-MRI) increases acquisition speed by simultaneously using multiple radio frequency (RF) detector coils.
This helps avoid some of the time-consuming phase-encoding steps in the MRI process. Although less k-space data is received for each coil, this is compensated by data being available from multiple coils. The first approach to p-MRI was based on an arrangement of $J$ surface coils around the object for MR imaging, one for each $k$-space line to be acquired \cite{HuRa}.
The p-MRI method to achieve routine use was sensitivity encoding, SENSE \cite{PrWeScBo, Pr}.
In this approach, a discrete Fourier transform is used to reconstruct an aliased image for each element in the array. Then the full full-of-view image is
generated from the individual sets of images.

Generally, the MRI signal acquired by receiver coil $j$ is given by
\begin{equation}
s_j(\vec{k}) = \int u(\vec{x})c_j(\vec{x})exp(i\vec{k}\cdot\vec{x})d\vec{x}, \ \ \  j = 1, 2, \ldots, J.\label{eq:1}
\end{equation}
Here $u$ is the excited proton density function, $c_j$ the sensitivity profile of the $j$th coil at $\vec{x}$,
and $\vec{k}$ is the chosen $k$-space trajectory. In discrete form \eqref{eq:1} can be written
\begin{equation}
s_j(k_m, k_n) = \sum_{m=1}^{N}\sum_{n=1}^{N}u(m, n)c_{j}(m, n)exp(ik_{m}m)exp(ik_{n}n), \ \ \  j = 1, 2, \ldots, J.\label{eq:2}
\end{equation}
If the coil sensitivity profiles $c_{j}(m, n), \ \  j = 1, 2, \ldots, J,$ are known, the system of equations \eqref{eq:2} can be numerically inverted for $u$ with relative ease \cite{HoBrMaKy, UeHoBlFr}. An early direct method to invert \eqref{eq:2} is to decouple the system of equations in image space under regular sub-sampled pattern like SENSE \cite{PrWeScBo, PrWeBoBo, KyPaKaWeBaMuJo}. Another direct approach is to approximate a sparse inverse by using the coil data in k-space as in SMASH \cite{SoMa} and g-SMASH \cite{ByLaHa}. SMASH is a partial p-MRI method using multiple coils to speedup acquisition in the course of imaging. Whereas g-SMATH is generalized SMASH method that reconstructs image with the coil data in k-space. However, the MRI signal equation will be increasingly ill-conditioned when the acceleration factor becomes large. The acceleration factor is the ratio of the amount of k-space data required for a fully sampled image to the amount collected in accelerated acquisition. When ill-conditioned, the inversion of the linear system \eqref{eq:2} will lead to the amplification of noise present in the MRI signal $s_l$. Therefore  regularization methods are required to improve reconstruction quality. Historically employed regularization methods include the truncated singular value decomposition (TSVD) and damped least-squares (DLS) \cite{LaNu}.

If the coil sensitivities are not known, it is common to acquire sensitivity information by using a calibration step \cite{GrBrBlKaHeMuNiJeKiJa}. For example, the coil sensitivity profiles can be obtained directly from the reference lines in autocalibrating SENSE \cite{McYeOhPrSo}. The GRAPPA method \cite{GrJaHeNiJeWaKiHa} is the most widely used autocalibrating technique in the determination of coil sensitivities. The coils sensitivities are generally determined from the center of the k-space rather than using all available information. Due to small errors, this leads to residual aliasing artifacts in the reconstruction because. Nonlinear inversion with the joint the estimation of the coil sensitivities $c_j$ and the determination of the proton density image $u$, can improve reconstruction quality \cite{BaKa, YiSh, UeHoBlFr, UeKaFr}.

\subsection{Nonlinear inversion for p-MRI}

Parallel MR imaging can be formulated as a nonlinear inverse problem with a nonlinear forward operator $\mathfrak{F}$, which maps the proton density $u$ and the coil sensitivities $c = (c_j, c_2, \ldots, c_J)^T$ to the measured k-space  data $g$ as
\begin{equation}
\mathfrak{F}(u, c) := \big(P\mathcal{F}(u\cdot c_1), P\mathcal{F}(u\cdot c_2), \ldots, P\mathcal{F}(u\cdot c_J)\big)^T = g. \label{eq:3}
\end{equation}
Here $P$ is the binary sub-sampling mask, $\mathcal{F}$ is the discrete 2D Fourier transform, and $g = (g_1, g_2, \ldots, g_J)^T$
the acquired k-space measurements for $J$ receiver coils. As shown in \cite{UeHoBlFr, KnClUeSt}, the problem \eqref{eq:3} can
be solved by the iteratively regularized Gauss-Newton (IRGN) method \cite{BaKo, EnHaNe, BlNeSc, Ho}. The discrepancy principle is used to obtain a suitable level of regularization. In \cite{KnClBrUeSt} the authors furthermore expanded IRGN method with variational regularization terms to improve reconstruction quality. The method works as follows. Writing $v = (u, c_1, c_2, \ldots, c_J)^T$, and starting from an initial guess $v^0$, we solve on each step for $\triangle v$ from the linearised problem
\begin{equation}
\min_{\triangle v}~\frac{1}{2} \big\|\mathfrak{F}'(v^k)\triangle v + \mathfrak{F}(v^k) - g\big\|_2^2 + \frac{\alpha_k}{2}R_{c}(c^k + \triangle c) + \beta_kR_{u}(u^k + \triangle u)  \label{eq:irgn}.
\end{equation}
Then we update $v^{k+1} := v^k +\triangle v$.
Here $R_{c}(c)$  is a regularization functional for penalizing the high Fourier coefficients of the coils $c_j, j = 1, 2, \ldots, J$, and $R_{u}$ regularizes the image. The regularization parameters $\alpha_k$ and $\beta_k$ are updated by the formulas $\alpha_{k+1} := q_{\alpha}\alpha_k$ and $\beta_{k+1} := q_{\beta}\beta_k$ with $0 < q_{\alpha}, q_{\beta} < 1.$ More details can be found in \cite{KnClBrUeSt}.

There are many options for the regularisers $R_u$ and $R_c$ in the inverse problems literature. The most basic regularization is the simple $L^2$ penalty $R_{u}(u) = \frac{1}{2} \|u\|_2^2$. This is used in \cite{UeHoBlFr, KnClUeSt, KnClBrUeSt}. Another conventional choice for the image $u$ is $R_{u}(u) = TV(u)$, the Total Variation \cite{RuOsFa, Ch, LiKwBeWa, HoBrMaKy1, ZhSh, ZhShZhYu}. The are two common variants of the total variation, dependent on the choice of pointwise norm used.
Restricting ourselves to the finite-dimensional setting, with the two-norm we obtain the isotropic total variation
\begin{equation}
\mathrm{TV}_{I}(u) = \sum_{m=1}^{N}\sum_{n=1}^{N}\sqrt{|\nabla_1u(m, n)|^2+|\nabla_2u(m, n)|^2},  \label{eq:5}
\end{equation}
while with the $1$-norm we obtain the computationally easier but anisotropic total variation
\begin{equation}
\mathrm{TV}_{l_1}(u) = \sum_{m=1}^{N}\sum_{n=1}^{N}|\nabla_1u(m, n)|+|\nabla_2u(m, n)|.  \label{eq:6}
\end{equation}
In both cases we have used the forward-differences
{\small\begin{align}
\nabla_1 u(m, n) & = \begin{cases}
u(m+1, n) - u(m, n), & m < N\\
0, & \textrm{$m = N$,}
\end{cases}
&
\nabla_2 u(m, n) & = \begin{cases}
u(m, n+1) - u(m, n), & n < N\\
0 & \textrm{$n = N$.}
\end{cases}
\end{align}}

If the penalty parameter $\beta_k$ becomes large, TV regularization will generate staircasing artefacts. This can be avoided through the use of second-order Total Generalized Variation (TGV) \cite{BrKuPo,KnBrPoSt,VaBrKn}.

For the regularization term $R_{c}(c)$, one choice from \cite{BaKo, EnHaNe, BlNeSc, Ho} is to take $R_{c}(c) = \|w\cdot\mathcal{F}c\|_2^2$, where $w$ is an weighting operator that penalizes high Fourier coefficients. It is well known that coil sensitivities are generally rather smooth functions that vary only slowly and do not have sharp edges. This supports the use of quadratic regularization of the gradients the Tikhonov-regularized model in \cite{BeKnScVa}. Specifically, instead of the iteratively regularized IRGN approach \eqref{eq:irgn}, the authors directly solve for $\hat v = (\hat{u}, \hat{c_1}, \hat{c_2}, \ldots, \hat{c_J})$ the variational model
\begin{equation}
\min_{v=(u, c_1, c_2, \ldots, c_J)^T}~ \frac{1}{2}\sum_{j=1}^J\big\|P\mathcal{F}(G(v))_j - g_j\big\|_2^2 + \alpha_0R_u(u) + R_c(c),  \label{eq:variational}
\end{equation}
where $G(v) = (uc_1, uc_2, \ldots, uc_J)^T$, $R_{u}(u) = \mathrm{TV}_I(u)$, and $R_{c}(c) = \sum_{j = 1}^{J}\alpha_j\|\nabla c_j\|_{2,2}$.

\subsection{Contributions}

From the optimization point of view, a difficulty with both models \eqref{eq:irgn} and \eqref{eq:variational} is the multiplicative nature of $G(v)$. It can cause $u$ and $c_j$ to compete against each other. Therefore, besides physical considerations, one goal in the design of the regularisers $R_u$ and $R_c$ would be to try to make $u$ and $c_j$ in some vague sense ``orthogonal'', to avoid this competition. One approach to such vague orthogonality is to force $u$ piecewise constant and $c_j$ smooth. This is roughly performed by the TV and $H^1$ regularisers in \cite{BeKnScVa}. Another approach is for $u$ and $c_j$ have very different sparsity structure. This is what we will do in this paper.

Specifically, we will assume that the coil sensitivities can be sparsely represented in a spherical function basis $\{f_l^{+}\}$, which we introduce in detail in \cref{sec:sph}. Then, with $c_j=\sum_{l=1}^L a_l^{(j)} f_l^+$, we will in the variational model \eqref{eq:variational} promote sparsity by taking
\begin{equation}
    \label{eq:reg}
    R_c(c) = \alpha R_a(a) \quad\text{for}\quad
    R_a(a) = \sum_{j=1}^J \sum_{l=1}^L |a_l^{(j)}|.
\end{equation}
Therefore, we consider the model
\begin{equation}
\min_{v=(u, a)}~ \frac{1}{2}\sum_{j=1}^J\big\|P\mathcal{F}(G(v))_j - g_j\big\|_2^2 + \alpha_0R_u(u) + \alpha R_a(a),  \label{eq:9}
\end{equation}
for an appropriate definition of $G$ that we provide in \cref{sec:new}.

In order to make this model practical, we propose in \cref{sec:the} an efficient approach to compute the spherical basis functions based on spherical Bessel functions of the first kind and spherical harmonics. For the computation of spherical Bessel functions, we develop recurrence formulas. Based on the recurrence formula, all other spherical Bessel functions are efficiently calculated via the first two Bessel functions. For the computation of the spherical harmonics, we also provide a means to efficiently compute the associated Legendre functions by establishing  in \cref{sec:the} a recurrence formula with only four terms.
In \cref{sec:new} we then present a numerical method for \eqref{eq:variational} with \eqref{eq:reg}, based on the alternating direction method of multipliers (ADMM), and study the practical reconstruction performance in \cref{sec:num}.

\section{Spherical basis function representation of coil sensitivities}
\label{sec:sph}

According to the principle of reciprocity \cite{Ho1, Zh, Ib, JiLiWeLiCr}, the coil sensitivity maps can be evaluated from transmit radio frequency field profiles $B_1^{+}$. We therefore start by briefly introducing the theory of $B_1^{+}$ fields. Let $\omega, \mu, \sigma$, and $\varepsilon$ denote the Larmour frequency, the magnetic permeability, the conductivity, and the dielectric permittivity of the material, respectively. The radio frequency (RF) field is denoted by $\vec{B}(\vec{r}) = (B_x(\vec{r}), B_y(\vec{r}), B_z(\vec{r}))^T$ with $\vec{B}(\vec{r}) \in \varmathbb{C}^3$ and $\vec{r} \in \varmathbb{R}^3$. In the positively rotating frame given
in \cite{Ho1, JiLiWeLiCr, SbHoLaLuVa}, the transmit RF field is
\begin{equation}
B_1^{+}(\vec{r}) \equiv \frac{{B_{x}(\vec{r}) + iB_{y}(\vec{r})}}{2}.  \label{eq:10}
\end{equation}
For the detailed introduction of positively rotating frame, see \cite{Ho1}.
These fields an be approximated \cite{De, KeCaRoWoWoBu, OcAt, WiBoPr} by
\begin{equation}
B_1^{+}(\vec{r}) \approx \sum_{l=1}^L a_lf_l^{+}(\vec{r}),  \label{eq:11}
\end{equation}
where $f_l^{+}(\vec{r})$ are the spherical basis functions, $L$ is a small natural number and $a_l$ are complex coefficients.
The overall magnetic field $\vec{B}$ can be reduced to the Helmholtz equation \cite{ChDK}
\begin{equation}
\nabla^2 \vec{B}(\vec{r}) + \zeta^2\vec{B}(\vec{r}) = 0,  \label{eq:12}
\quad{where}\quad \zeta^2 = \varepsilon \mu \omega^2 - i\sigma\omega\mu.
\end{equation}

In spherical coordinates $(\rho, \theta, \phi)$, the equation \eqref{eq:12} has the solution \cite{Jo}
\begin{align}
\label{eq:13}
B_x(\vec{r}) & = \sum_{n=0}^\infty\sum_{m=-n}^n \alpha_n^mf_n^m(\vec{r}),
&
B_y(\vec{r}) & = \sum_{n=0}^\infty\sum_{m=-n}^n \beta_n^mf_n^m(\vec{r}),
&
B_z(\vec{r}) & = \sum_{n=0}^\infty\sum_{m=-n}^n \gamma_n^mf_n^m(\vec{r}),
\end{align}
where $f_n^m$ are so-called spherical functions. They can be written
\begin{equation}
    \label{eq:spherical}
    f_n^m(\rho, \theta, \phi) \equiv  j_n(\zeta\rho)Y_n^m(\theta, \phi),
\end{equation}
where $j_n$ is the spherical Bessel function of the first kind of order $n$, and $Y_n^m$ is the spherical harmonic of order $n$ and degree $m$.
The spherical functions form a basis for the $B_1^{+}$ fields by setting
\begin{equation}
    \label{eq:spherical-lplus}
    f_l^{+}(\rho, \theta, \phi) = f_n^m(\rho, \theta, \phi)
    \quad\text{with}\quad
    l = n^2 +n + m + 1, \ |m| \leq n \text{ and }0 \leq n \leq \tilde{n}.
\end{equation}
When signals or objects of approximately spherical shape are considered, fast convergence is expected, so the complex coefficients $\alpha_n^m$,
$\beta_n^m$, and $\gamma_n^m$ in \eqref{eq:13} should be negligible for $n > \tilde{n}$ with $\tilde{n}$ being a small natural number. We can therefore also expect fast convergence for the spherical function approximation of the $B_1^{+}$ field, and by extension the coil sensitivities.

\section{Efficient computation of the spherical basis functions}
\label{sec:the}

Our task in the present section is develop efficient recurrence formulas for the computation of the spherical basis functions \eqref{eq:spherical}. As discussed, this will be based on formulas for the Bessel functions and spherical harmonics.

\subsection{A recurrence relation for the spherical Bessel functions}
Following \cite{ArWe}, we now develop a recurrence formula for the spherical Bessel functions $j_n$.
We start by recalling that the Bessel function of the first kind, for arbitrary order $\alpha \in \varmathbb{R}$, is defined as
\begin{equation}
  J_\alpha(x) = \sum_{s=0}^\infty{{(-1)^s} \over {s!\Gamma(s+\alpha+1)!}}\bigg(\frac{x}{2}\bigg)^{\alpha+2s}.  \label{eq:19}
\end{equation}
Since the $\Gamma$ satisfies $\Gamma(n)=(n-1)!$ for integral $n \ge 0$ we can in particular write
\begin{equation}
J_n(x) = \sum_{s=0}^\infty{{(-1)^s} \over {s!(n+s)!}}\bigg(\frac{x}{2}\bigg)^{n+2s}.  \label{eq:19}
\end{equation}
To compute $J_n$ for negative integers, we can use the relationship
\begin{equation}
J_{-n}(x) = (-1)^nJ_n(x).  \label{eq:22}
\end{equation}
In addition, we have the recurrence relationship \cite{ArWe}
\begin{equation}
J_{n-1}(x) + J_{n+1}(x) = {{2n} \over x}J_n(x).  \label{eq:23}
\end{equation}

We now finally define the spherical Bessel function
\begin{equation}
    \label{eq:24}
    j_n(x) = \sqrt{{\pi \over {2x}}}J_{n+1/2}(x).
\end{equation}
From the recurrence relation \eqref{eq:23} for the Bessel functions $J_n$, we then obtain
\begin{equation}
j_{n-1}(x)+j_{n+1}(x) = {{2n+1} \over x}j_n(x).
  \label{eq:32}
\end{equation}
This can more conveniently be rewritten
\begin{equation}
j_{n+1}(x) = {{2n+1} \over x}j_n(x) - j_{n-1}(x).
  \label{eq:33}
\end{equation}
This is a three-term recurrence relation, so if $j_0$ and $j_1$ are known, then any higher-order $j_n$ can be computed from \eqref{eq:33}.

To compute $j_0$, we recall that Legendre's duplication formula for the $\Gamma$ function states
\begin{equation}
\Gamma(1+z)\Gamma(z+\frac{1}{2}) = 2^{-2z}\sqrt{\pi}\Gamma(2z+1).  \label{eq:25}
\end{equation}
For integral $z$ therefore
\begin{equation}
z!(z+\frac{1}{2})! = 2^{-2z-1}\sqrt{\pi}(2z+1)!.  \label{eq:26}
\end{equation}
Consequently
\begin{equation}
j_n(x) = 2^n x^n\sum_{s=0}^\infty{{(-1)^s(s+n)!} \over {s!(2s+2n+1)!}}x^{2s}.  \label{eq:27}
\end{equation}
When $n = 0$, we find from \eqref{eq:27} that
\begin{equation}
j_0(x) = \sum_{s=0}^\infty{{(-1)^s} \over {(2s+1)!}}x^{2s} = {{\sin x} \over x}.  \label{eq:j0}
\end{equation}

To compute $j_1$, we set
\begin{align}
    n_n(x) & =  (-1)^{n+1}\sqrt{{\pi \over {2x}}}J_{-n-1/2}(x), \quad\text{and} &
    h_n^{(1)}(x) & = j_n(x)+in_n(x).
\end{align}
In \cite{ArWe} it is shown
\begin{equation}
h_n^{(1)}(x) = (-i)^{n+1}{{e^{ix} \over x}}\sum_{s=0}^n{{i^s} \over {s!(2x)^s}}{{(n+s)!} \over {(n-s)!}},
  \label{eq:29}
\end{equation}
which for $n=1$ gives
\begin{equation}
h_1^{(1)}(x) = e^{ix}\bigg(-{1 \over x}-{i \over x^2}\bigg).
  \label{eq:30}
\end{equation}
As the real part of $h_1^{(1)}(x)$, we then obtain for $j_1(x)$ the expression
\begin{equation}
j_1(x) = {{\sin x} \over {x^2}}-{{\cos x} \over x}.
  \label{eq:j1}
\end{equation}

Based on the recurrence \eqref{eq:33} and the expressions \eqref{eq:j0} and \eqref{eq:j1} for $j_0$ and $j_1$, we can now compute any higher-order $j_n$.
We list the first few in \cref{tab:table1}.

\begin{table}[tbp]
  \caption{The first spherical Bessel functions}
  \label{tab:table1}
  \centering
  \begin{tabular}{c|l} \hline
  $n$ & $j_n(x)$\\ \hline
  $n = 0$ & $j_0(x) = {{\sin x} \over x}$  \\
  $n = 1$ & $j_1(x) = -{{\cos x} \over x}+{{\sin x} \over {x^2}}$  \\
  $n = 2$ & $j_2(x) = \bigg(-{1 \over x}+{3 \over {x^3}}\bigg)\sin x-{3 \over x^2}\cos x$ \\
  $n = 3$ & $j_3(x) = \bigg(-{15 \over x^4}-{6 \over {x^2}}\bigg)\sin x+\bigg({1 \over x}-{15 \over {x^3}}\bigg)\cos x$ \\ \hline
  \end{tabular}
\end{table}

\subsection{Efficient computation of the associated Legendre functions}

The spherical harmonics $Y_n^m$ are defined in spherical coordinates as \cite{ArWe}
\begin{equation}
Y_n^m(\theta, \phi) \equiv (-1)^m \sqrt{{{2n+1} \over {4\pi}}{(n-m)! \over (n+m)!}}P_n^m(\cos \theta)e^{im\phi},
\qquad
0 \leq \theta \leq \pi, 0 \leq \phi \leq 2\pi,
  \label{eq:34}
\end{equation}
where the associated Legendre function
\begin{equation}
P_n^m(x) = (1-x^2)^{m/2}{d^m \over {dx^m}}P_n(x),
\qquad 0 \leq m \leq n,
  \label{eq:35}
\end{equation}
and $P_n(x)$ are the $n$th-order Legendre polynomials.
They are defined as
\begin{equation}
P_n(x) = \sum_{k=0}^{\lfloor n/2\rfloor}(-1)^k{{(2n-2k)!} \over {2^nk!(n-k)!(n-2k)!}}x^{n-2k}
  \label{eq:36}
\end{equation}
with $\lfloor n/2\rfloor = n/2$ for $n$ even, $(n-1)/2$ for $n$ odd.
In particular, it is easy to see that
\[
    P_0(x) = 1 \quad\text{and}\quad P_1(x) = x.
\]

For $-n \leq m < 0$, using Leibniz' differentiation formula, we can find that $P_n^m(x)$ and $P_n^{-m}(x)$ are related by \cite{ArWe}
\begin{equation}
P_n^{-m}(x) = (-1)^m{{(n-m)!} \over {(n+m)!}}P_n^m(x).
  \label{eq:42}
\end{equation}
For $|m| > n$, $P_n^m(x) = 0$.

From \eqref{eq:35}, we have
\begin{equation}
P_n^0(x) = P_n(x).
  \label{eq:43}
\end{equation}
Thus $P_0^0(x) = P_0(x) = 1$, $P_1^0(x) = P_1(x) = x$, $P_1^1(x) =(1-x^2)^{1/2}{d \over {dx}}P_1(x) = (1-x^2)^{1/2}$.
By \eqref{eq:42}, we obtain $P_1^{-1}(x) = -\frac{1}{2}(1-x^2)^{1/2}$.
For the sake of the convenience of developing the recurrence relation to find all the $P_n^m$ effectively, we put $P_1^{-1}, P_0^0, P_1^0, P_1^1$ together as
\begin{align}
\label{eq:44}
P_0^0(x) & = 1, &
P_1^{-1}(x) & = -\frac{1}{2}(1-x^2)^{1/2}, &
P_1^0(x) & =  x, &
P_1^1(x) & = (1-x^2)^{1/2}.
\end{align}

Let us define the polynomials $\mathcal{P}_{s+m}^m(x) = P_{s+m}^m(x)(1-x^2)^{-m/2}$, $m\geq 0$. Then the generating function \cite[(12.83)]{ArWe}
\begin{equation}
\bar{g}_m(x, t) \equiv {{(2m)!} \over {2^mm!(1-2tx+t^2)^{m+1/2}}} = \sum_{s=0}^\infty \mathcal{P}_{s+m}^m(x)t^s.
  \label{eq:45}
\end{equation}
Furthermore, following \cite[§12.5]{ArWe}, we have the recursion
\begin{equation}
\mathcal{P}_{s+m}^{m}(x) = 2x\mathcal{P}_{s+m-1}^{m}(x)-\mathcal{P}_{s+m-2}^{m}(x)+(2m-1)\mathcal{P}_{s+m-1}^{m-1}(x).
  \label{eq:47}
\end{equation}
Therefore
\begin{equation}
    \label{eq:psm-deriv}
    \begin{split}
     P_{s+m}^m(x)
     & = (1-x^2)^{m/2}\mathcal{P}_{s+m}^m(x)
     \\
     & = (1-x^2)^{m/2}\big(2x\mathcal{P}_{s+m-1}^m(x)-\mathcal{P}_{s+m-2}^m(x)+(2m+1)\mathcal{P}_{s+m-1}^{m-1}(x)\big)
     \\
     & =  2(1-x^2)^{m/2}x\mathcal{P}_{s+m-1}^m(x)-(1-x^2)^{m/2}\mathcal{P}_{s+m-2}^m(x)
     \\
     & \qquad +(2m-1)(1-x^2)^{m/2}\mathcal{P}_{s+m-1}^{m-1}(x)
     \\
     & =  2(1-x^2)^{m/2}x(1-x^2)^{-m/2}P_{s-1+m}^m(x)-(1-x^2)^{m/2}(1-x^2)^{-m/2}P_{s-2+m}^m(x)
     \\
     & \qquad +(2m-1)(1-x^2)^{m/2}(1-x^2)^{-(m-1)/2}P_{s+m-1}^{m-1}(x)
     \\
     & = 2xP_{s+m-1}^m(x)-P_{s+m-2}^m(x)+(2m-1)(1-x^2)^{1/2}P_{s+m-1}^{m-1}(x).
    \end{split}
\end{equation}
That is
\begin{equation}
    \label{eq:48}
    P_{s+m}^m(x) = 2xP_{s+m-1}^m(x)-P_{s+m-2}^m(x)+(2m-1)(1-x^2)^{1/2}P_{s+m-1}^{m-1}(x).
\end{equation}
Now, for $m \geq 0$, using \eqref{eq:48}, we can compute effectively all the $P_n^m(x)$ by starting with \eqref{eq:44}.
For $m < 0$, we can then use \eqref{eq:42}.
We list the first few Legendre functions are listed in \cref{tab:table2}.
Recalling \eqref{eq:34}, we are in particular interested int the case $x = \cos \theta$, which we also list.

\begin{table}[tbp]
  \caption{First few associated Legendre functions as functions of $x$ and of $x=\cos\theta$.}
  \label{tab:table2}
  \centering
  \small\setlength{\tabcolsep}{4pt}
  \begin{tabular}{lll|lll}
  $P_n^m$ & fn.~of $x$ & fn.~of $\theta$ &
  $P_n^m$ & fn.~of $x$ & fn.~of $\theta$ \\
  \hline
  $P_0^0(x)$ & 1 & 1  &
  $P_1^{-1}(x)$ & $-\frac{1}{2}(1-x^2)^{1/2}$ & $-\frac{1}{2}\sin \theta$  \\
  $P_1^0(x)$ & $x$ & $\cos \theta$ &
  $P_1^1(x)$ & $ (1-x^2)^{1/2}$ & $\sin \theta$ \\
  $P_2^{-2}(x)$ & $ {1 \over 8}(1-x^2)$ & ${1 \over 8}\sin^2 \theta$ &
  $P_2^{-1}(x)$ & $ -\frac{1}{2}x(1-x^2)^{1/2}$ & $-\frac{1}{2}\cos \theta \sin \theta$ \\
  $P_2^0(x)$ & ${3 \over 2}x^2 - \frac{1}{2}$ & ${3 \over 2}\cos^2 \theta-\frac{1}{2}$  &
  $P_2^1(x)$ & $3x(1-x^2)^{1/2}$ & $3\cos \theta\sin \theta$  \\
  $P_2^2(x)$ & $ 3(1-x^2)$ & $3\sin^2\theta$ &
  $P_3^{-3}(x)$ & $ -{1 \over 48}(1-x^2)^{3/2}$ & $-{1 \over 48}\sin^3\theta$ \\
  $P_3^{-2}(x)$ & $ {1 \over 8}x(1-x^2)$ & ${1 \over 8}\cos \theta\sin^2\theta$ &
  $P_3^{-1}(x)$ & $-{1 \over 8}(5x^2-1)(1-x^2)^{1/2}$ & ${1 \over 8}(5\cos^2\theta-1)\sin\theta$  \\
  $P_3^0(x)$ & $\frac{1}{2}x(5x^2-3)$ & $\frac{1}{2}\cos\theta(5\cos^2\theta-3)$ &
  $P_3^1(x)$ & $ {3 \over 2}(5x^2-1)(1-x^2)^{1/2}$ & ${3 \over 2}(5\cos^2\theta-1)\sin\theta$ \\
  $P_3^2(x)$ & $ 15x(1-x^2)$ & $15\cos\theta\sin^2\theta$ &
  $P_3^3(x)$ & $ 15(1-x^2)^{3/2}$ & $15\sin^3\theta$ \\  \hline
  \end{tabular}
\end{table}

Using the recurrence \eqref{eq:48} for $P_n^m(\cos \theta)$, and the explicit solutions in \cref{tab:table2}, we now easily find all the spherical harmonics $Y_n^m(\theta, \phi)$ by the formula \eqref{eq:34}.
\Cref{tab:table3} shows some of the low-order ones.

\begin{table}[tbp]
  \caption{Some low-order spherical harmonics}
  \label{tab:table3}
  \centering
  \begin{tabular}{ll} \hline
  $Y_0^0(\theta, \phi) = {1 \over {\sqrt{4\pi}}}$  &
  $Y_1^{-1}(\theta, \phi) = \sqrt{{3 \over {8\pi}}}\sin \theta e^{-i\phi}$  \\
  $Y_1^0(\theta, \phi) =  \sqrt{{3 \over {4\pi}}}\cos \theta $ &
  $Y_1^1(\theta, \phi) =  -\sqrt{{3 \over {8\pi}}}\sin \theta e^{i\phi}$ \\
  $Y_2^{-2}(\theta, \phi) =  \sqrt{{15 \over {32\pi}}}\sin^2 \theta e^{-2i\phi}$ &
  $Y_2^{-1}(\theta, \phi) =  \sqrt{{15 \over {8\pi}}}\sin \theta\cos \theta e^{-i\phi}$ \\
  $Y_2^0(\theta, \phi) = \sqrt{{5 \over {4\pi}}}\bigg({3 \over 2}\cos^2\theta-\frac{1}{2}\bigg)$  &
  $Y_2^1(\theta, \phi) = -\sqrt{{5 \over {24\pi}}}3\sin \theta\cos \theta e^{i\phi}$  \\
  $Y_2^2(\theta, \phi) =  \sqrt{{15 \over {32\pi}}}\sin^2 \theta e^{i2\phi}$ &
  $Y_3^{-3}(\theta, \phi) =  \sqrt{{35 \over {64\pi}}}\sin^3 \theta e^{-i3\phi}$ \\
  $Y_3^{-2}(\theta, \phi) =  \sqrt{{105 \over {32\pi}}}\cos \theta\sin^2 \theta e^{-i2\phi}$ &
  $Y_3^{-1}(\theta, \phi) = \sqrt{{21 \over {64\pi}}}(5\cos^2 \theta-1)\sin \theta e^{-i\phi}$  \\
  $Y_3^0(\theta, \phi) = \sqrt{{7 \over {16\pi}}}\cos \theta(5\cos^2\theta-3)$  &
  $Y_3^1(\theta, \phi) =  -\sqrt{{21 \over {64\pi}}}(5\cos^2\theta-1)\sin\theta e^{i\phi}$ \\
  $Y_3^2(\theta, \phi) =  \sqrt{{105 \over {32\pi}}}\cos \theta\sin^2\theta e^{i2\phi}$ &
  $Y_3^3(\theta, \phi) =  -\sqrt{{35 \over {64\pi}}}\sin^3\theta e^{i3\phi}$ \\  \hline
  \end{tabular}
\end{table}

\subsection{A method to compute the spherical basis functions}

We recall te presentation \label{eq:spherical-lplus} of $f_l^+$ and the presentation \eqref{eq:spherical} of $f_n^m$.
Because $l = n^2+n+m+1$ and $|m| \leq n$, we find that $l$ is bounded by $l_{\max}^{(n)} = (n+1)^2$ for $n$. From $f_l^{+}(\rho, \theta, \phi) = f_n^m(\rho, \theta, \phi)$, we easily get for the low-order functions the relantionships in \cref{tab:table4}. Using $f_n^m(\rho, \theta, \phi) = j_n(\zeta \rho)Y_n^m(\theta, \phi)$, we then obtain the basis functions $f_l^{+}$.
\Cref{tab:table5} lists some $f_l^{+}$ corresponding to low-order $n$.

\begin{table}[tbp]
  \caption{The relationship between $f_l^{+}$ and $f_n^m$}
  \label{tab:table4}
  \centering\small
  \begin{tabular}{c|c|c} \hline
  $n$ & $f_n^m$ & $f_l^{+}$ ($l_{\max}^{(n)})$\\ \hline
  $n = 0$ & $f_0^0$  & $f_1^{+} (l_{\max} = (0+1)^2 = 1)$ \\
  $n = 1$ & $f_1^{-1}, f_1^0, f_1^1$ & $f_2^{+}, f_3^{+}, f_4^{+} ~(l_{\max}^{(n)} = (1+1)^2 = 4)$\\
  $n = 2$ & $f_2^{-2}, f_2^{-1}, f_2^0, f_2^1, f_2^2$ & $f_5^{+}, f_6^{+}, f_7^{+}, f_8^{+}, f_9^{+} ~(l_{\max}^{(n)} = (2+1)^2 = 9)$\\
  $n = 3$ & $f_3^{-3}, f_3^{-2}, f_3^{-1}, f_3^0, f_3^1, f_3^2, f_3^3$ & $f_{10}^{+}, f_{11}^{+}, f_{12}^{+}, f_{13}^{+}, f_{14}^{+}, f_{15}^{+}, f_{16}^{+} ~(l_{\max}^{(n)} = (3+1)^2 = 16)$\\
  \vdots & \vdots & \vdots\\ \hline
  \end{tabular}
\end{table}

\begin{table}[tbp]
  \caption{The the basis functions $f_l^{+}$ corresponding to order $n$}
  \label{tab:table5}
  \centering
  \begin{tabular}{l|l} \hline
  $n$ &  $f_l^{+} $\\ \hline
  $n = 0$ & $f_1^{+} = j_0(\zeta\rho)Y_0^0(\theta, \phi)$   \\ \hline
  $n = 1$ & $f_2^{+} = j_1(\zeta\rho)Y_1^{-1}(\theta, \phi), f_3^{+} = j_1(\zeta\rho)Y_1^0(\theta, \phi), f_4^{+} = j_1(\zeta\rho)Y_1^1(\theta, \phi)$ \\ \hline
  $n = 2$ & $f_5^{+} = j_2(\zeta\rho)Y_2^{-2}(\theta, \phi), f_6^{+} = j_2(\zeta\rho)Y_2^{-1}(\theta, \phi), f_7^{+} = j_2(\zeta\rho)Y_2^0(\theta, \phi), $\\
   & $f_8^{+} = j_2(\zeta\rho)Y_2^1(\theta, \phi), f_9^{+} = j_2(\zeta\rho)Y_2^2(\theta, \phi)$ \\ \hline
  $n = 3$ & $f_{10}^{+} = j_3(\zeta\rho)Y_3^{-3}(\theta, \phi), f_{11}^{+} = j_3(\zeta\rho)Y_3^{-2}(\theta, \phi), f_{12}^{+} = j_3(\zeta\rho)Y_3^{-1}(\theta, \phi),$ \\ & $f_{13}^{+} = j_3(\zeta\rho)Y_3^0(\theta, \phi),$
    $f_{14}^{+} = j_3(\zeta\rho)Y_3^1(\theta, \phi), f_{15}^{+} = j_3(\zeta\rho)Y_3^2(\theta, \phi)$, \\ & $f_{16}^{+} = j_3(\zeta\rho)Y_3^3(\theta, \phi)$ \\ \hline
  \vdots & \vdots \\ \hline
  \end{tabular}
\end{table}

From the above analysis, we find that if $j_n$ and $P_n^m$ are given only for $n =0, 1$, then rather than directly evaluating the series \eqref{eq:27} and \eqref{eq:36}, we can quickly find all the basis functions $f_l^{+}$ from \eqref{eq:33} and \eqref{eq:48}.
Starting with $j_0(\zeta\rho) = {{{\sin (\zeta\rho)} \over {\zeta\rho}}}$, $j_1(\zeta\rho) = {{{\sin (\zeta\rho)} \over (\zeta\rho)^2}} - {{{\cos (\zeta\rho)} \over {\zeta\rho}}}$, $P_0^0(\cos \theta) =1$, $P_1^{-1}(\cos \theta) = - {\frac{1}{2}}\sin \theta$, $P_1^0(\cos \theta) = \cos \theta$, $P_1^1(\cos \theta) = \sin \theta$, we can compute all the $f_l^{+}$ with $|m| \leq n, l = 1, 2, \ldots, l_{\max}^{(n)}$.
The numerical computation is outlined in \cref{alg:buildtree}. In its practical application, the radius $\rho$, the polar angle $\theta$ and the azimuth angle $\phi$  are computed by Cartesian coordinates $(x, y, z)$, which are as follows:
$$\rho_{i, j} = \sqrt{x_i^2+y_j^2+z_0^2},$$
$$\theta_{i, j} = \arccos({z_0 \over \rho_{i, j}}),$$
$$\phi_{i, j} = \arctan({y_j \over x_i}),$$
where $z_0$ is fixed, $x_i, y_j$ are discrete results of $x, y$, $i, j = 1, 2, \ldots, N$.
Hence
\[
    \left(f_l^{+}\right)_{i, j} = j_n(\zeta\rho_{i, j})Y_n^m(\theta_{i, j}, \phi_{i, j}) \quad\text{with}\quad l = n^2 +n +m +1, |m| \leq n.
\]
The $B_1^{+}$ field approximation \eqref{eq:11} can be expressed in matrix-vector form $B_1^{+} \approx FA$, where for $\ell = l_{\max}^{(\tilde{n})}$ we have
\[
F =
\left(
\begin{smallmatrix}
({f_1^+})_{1,1} & ({f_2^+})_{1,1} & \cdots & ({f_{\ell}^+})_{1,1} \\
({f_1^+})_{2,1} & ({f_2^+})_{2,1} & \cdots & ({f_{\ell}^+})_{2,1}\\
\vdots & \vdots &\ddots & \vdots  \\
({f_1^+})_{N,1} & ({f_2^+})_{N,1} & \cdots & ({f_{\ell}^+})_{N,1}\\
({f_1^+})_{1,2} & ({f_2^+})_{1,2} & \cdots & ({f_{\ell}^+})_{1,2}\\
({f_1^+})_{2,2} & ({f_2^+})_{2,2} & \cdots & ({f_{\ell}^+})_{2,2}\\
\vdots & \vdots &\ddots & \vdots  \\
({f_1^+})_{N,2} & ({f_2^+})_{N,2} & \cdots & ({f_{\ell}^+})_{N,2}\\
\vdots & \vdots &\ddots & \vdots  \\
({f_1^+})_{1,N} & ({f_2^+})_{1,N} & \cdots & ({f_{\ell}^+})_{1,N}\\
({f_1^+})_{2,N} & ({f_2^+})_{2,N} & \cdots & ({f_{\ell}^+})_{2,N}\\
\vdots & \vdots &\ddots & \vdots  \\
({f_1^+})_{N,N} & ({f_2^+})_{N,N} & \cdots & ({f_{\ell}^+})_{N,N}
\end{smallmatrix}
\right),
\quad\text{and}\quad
A =
\begin{pmatrix}
a_1 \\
a_2 \\
\vdots \\
a_{l_{\max}^{(\tilde{n})}}
\end{pmatrix}.
\]

\begin{algorithm}
\caption{The computation for the  basis functions $f_l^{+}$}
\label{alg:buildtree}
\begin{algorithmic}
\STATE{\textbf{Given} $\omega, \mu, \sigma$, $\varepsilon$, $\rho, \theta$, $\phi$ and $\tilde{n}$.}
\STATE{\textbf{Calculate} $\zeta = \sqrt{\varepsilon\mu\omega^2-i\sigma\omega\mu}$.}
\STATE{\textbf{Initialize} $j_0(\zeta\rho) = {{\sin(\zeta\rho)} \over {\zeta\rho}}$, $j_1(\zeta\rho) = {{{\sin (\zeta\rho)} \over (\zeta\rho)^2}} - {{{\cos (\zeta\rho)} \over {\zeta\rho}}}$, }
\STATE{\qquad \qquad \quad \ $P_0^0(\cos \theta) =1$, $P_1^{-1}(\cos \theta) = - {\frac{1}{2}}\sin \theta$,}
\STATE{\qquad \quad \qquad \ $P_1^0(\cos \theta) = \cos \theta$, $P_1^1(\cos \theta) = \sin \theta$ }
\STATE{\qquad \qquad \quad \ $Y_0^0(\theta, \phi) ={1 \over {\sqrt{4\pi}}}$, $Y_1^{-1}(\theta, \phi) = \sqrt{{3 \over {8\pi}}}\sin \theta e^{-i\phi}$,}
\STATE{\qquad \quad \qquad \ $Y_1^0(\theta, \phi) = \sqrt{{3 \over {4\pi}}}\cos \theta$, $Y_1^1(\theta, \phi) = -\sqrt{{3 \over {8\pi}}}\sin \theta e^{i\phi}$. }
\STATE{\textbf{Compute} $f_1^{+} = j_0(\zeta\rho)Y_0^0(\theta, \phi)$, $f_2^{+} = j_1(\zeta\rho)Y_1^{-1}(\theta, \phi),$}
\STATE{\qquad \qquad \quad \ $f_3^{+} = j_1(\zeta\rho)Y_1^0(\theta, \phi), f_4^{+} = j_1(\zeta\rho)Y_1^1(\theta, \phi)$.}
\STATE{\textbf{For} $n = 2:\tilde{n}$}
\STATE{\qquad \qquad \quad \ \textbf{$\bullet$ calculate} $j_n(\zeta\rho)$ and $P_n^m(\cos \theta)$ using \eqref{eq:33} and \eqref{eq:48}, respectively.}
\STATE{\qquad \qquad \quad \ \textbf{$\bullet$ calculate} $Y_n^m(\theta, \phi)$ with \eqref{eq:34}}.
\STATE{\qquad \qquad \quad \ \textbf{$\bullet$ calculate} $f_l^{+} = j_n(\zeta\rho)Y_n^m(\theta, \phi)$ with $|m| \leq n$ and
$l=n^2+n+m+1$.}
\STATE{\textbf{End}}
\STATE{\textbf{Return} $\{f_l^{+} \mid l=1,2,\ldots,l_{\max}^{(\tilde n)}\}$.}
\end{algorithmic}
\end{algorithm}

\section{The new regularization model and its numerical realization}
\label{sec:new}

We now present in detail our spherical function based regularization model for p-MRI reconstruction, as well as a method for its numerical realization.

\subsection{Regularization by sparse presentation in spherical basis}

Replacing the coil sensitivities $c_j$ and their regularization $R_c$ in the model \eqref{eq:variational} by the spherical functions representation $\sum_{l=1}^{l_{\max}^{(\tilde{n})}} a_l^{(j)}{f_l^{+}}$, $j =1, 2, \ldots, J$, and $R_a(a) = \alpha\sum_{j=1}^J\sum_{l=1}^{l_{\max}^{(\tilde{n})}}|a_l^{(j)}|$, respectively,  we obtain the model \eqref{eq:9}, that is
\begin{equation}
\min_{v = (u; a)}
\frac{1}{2}\sum_{j=1}^J\big\|P\mathcal{F}(G(v))_j - g_j\big\|_2^2 + \alpha_0R_u(u) + \alpha R_a(a),  \label{eq:49}
\end{equation}
where
\begin{equation}
G(v) = \bigg(u\sum_{l=1}^{l_{\max}^{(\tilde{n})}} a_l^{(1)}{f_l^{+}}, u\sum_{l =1}^{l_{\max}^{(\tilde{n})}} a_l^{(2)}{f_l^{+}}, \ldots, u\sum_{l =1}^{l_{\max}^{(\tilde{n})}} a_l^{(J)}{f_l^{+}}\bigg)^T,
\quad
R_a(a) = \sum_{j=1}^J\sum_{l=1}^{l_{\max}^{(\tilde{n})}}\big|a_l^{(j)}\big|,
\label{eq:50}
\end{equation}
and
\begin{equation}
a = \bigg(a_1^{(1)}, a_2^{(1)}, \ldots, a_{l_{\max}^{(\tilde{n})}}^{(1)},
a_1^{(2)}, a_2^{(2)}, \ldots, a_{l_{\max}^{(\tilde{n})}}^{(2)}, \ldots, a_1^{(J)}, a_2^{(J)}, \ldots, a_{l_{\max}^{(\tilde{n})}}^{(J)}\bigg)^T
\in \varmathbb{C}^{Jl_{\max}^{(\tilde{n})}}
. \label{eq:51}
\end{equation}
For the image $u$, we use the same isotropic total variation regularization $R_u(u)=\mathrm{TV}_I(u)$ as in the model \eqref{eq:variational} from \cite{BeKnScVa}, while our $R_a$ yields a different regularization $R_c$ for the coil sensitivities.

The complex numbers $a_l^{(j)}$ are the coefficients corresponding to the spherical function representation. The spherical basis functions are smooth enough to mainly encode low-frequency information for low $l$, so should not pick up important image features. Apart from this, we will not need to store the coil sensitives, which are themselves relatively high-dimensional images, or differentiate and perform their updates in a numerical optimization algorithm, as would be the case with the mode \eqref{eq:9} in \cite{BeKnScVa}. Instead we work with the relatively low-dimensional $a$.
Thus the updated model can be expected to be effective for parallel MRI reconstruction. We will validate this with the numerical experiments in \cref{sec:num}, but now we need to construct the optimization algorithm to solve the model \eqref{eq:49}, that is \eqref{eq:9}.

\subsection{The ADMM for convex problems}

We now start building a method for solving variational problems of type \eqref{eq:9}, i.e., \eqref{eq:49}. Note that due to the structure of $G$, these problems are non-convex. We therefore follow the non-linear ADMM approach of \cite{BeKnScVa}, itself motivated by the non-linear primal--dual hybrid gradient method (modified; PDHGM) of \cite{Va1,tuomov-pdex2nlpdhgm}. Here we simplify the derivations from \cite{BeKnScVa} to our specific problem form, and squared Hilbert space distances in place of the general Bregman distances employed in \cite{BeKnScVa}.

To motivate the algorithm that we will use, we  start by considering convex problems
\begin{equation}
\inf_{v\in V} \{E(v) + F(Bv)\},  \label{eq:52}
\end{equation}
where $E: V \longrightarrow \  ]-\infty, +\infty]$ and $F: H\longrightarrow \ ]-\infty, +\infty]$ are two proper lower semi-continuous convex functions, $B$ is a linear operator from $V$ into $H$, and $V$ and $H$ are the real Hilbert spaces equipped with the inner products $\langle\cdot, \cdot\rangle_V$ and $\langle \cdot, \cdot\rangle_H$.
The augmented dual problem for \eqref{eq:52} is
\begin{equation}
\sup_{\mu} \varphi_{\delta}(\mu) := \inf_{v\in V}\bigg\{E(v)+\inf_{q \in H}\bigg(F(q) + \langle\mu, Bv-q\rangle_H+{{\delta \over 2}}\parallel Bv-q\parallel_H^2\bigg)\bigg\}.  \label{eq:54}
\end{equation}
where, $\delta >0$, and $\varphi_{\delta}: H \longrightarrow [-\infty, +\infty[$.
We can solve \eqref{eq:52} and the dual \eqref{eq:54} by finding on $V\times H\times H$ a saddle point of the augmented Lagrangian function defined by
\begin{equation}
\mathscr{L}_{\delta}(v, q, \mu) = F(q) + E(v) +\langle\mu, Bv-q\rangle_H + {{\delta} \over 2}\parallel Bv-q\parallel_H^2. \label{eq:55}
\end{equation}

\begin{theorem}\label{thm:bigthm}
  If $\{u, p, \lambda\}$ is a saddle-point of $\mathscr{L}_{\delta}(v, q, \mu)$ on $V \times H \times H$ for $\delta > 0$, then
  $u$ is a solution of \eqref{eq:52}, and we have $p = Bu$.
\end{theorem}
\begin{proof}
  The proof is a direct consequence of Theorem 2.1 in Chapter III \cite{GaD}.
\end{proof}
In view of \cref{thm:bigthm}, in order to calculate the saddle-points of $\mathscr{L}_{\delta}(v, q, \mu)$ on
$V \times H \times H$, we employ an algorithm of the Uzawa type:
given $\lambda^{0} \in H$, determine $\{u^{k+1}, p^{k+1}\}$ from
\begin{equation}
\min_{\{v, q\}\in {V\times H}} \  \mathscr{L}_{\delta}(v, q, \lambda^k),
\label{eq:57}
\end{equation}
and then update
\begin{equation}
\lambda^{k+1} = \lambda^k + \delta(Bu^{k+1} - p^{k+1}).
\label{eq:58}
\end{equation}
For more details we refer to \cite{ArHuUz}.
The direct solution to \eqref{eq:57} is due to the coupling of $v$ and $q$.
For this reason, we are led to the Alternating Direction Method of Multipliers (ADMM) algorithm introduced in \cite{ArHuUz} to solve \eqref{eq:57}.
This algorithm approximates the pair $(u^{k+1}, p^{k+1})$ for \eqref{eq:57} via decoupled minimization over $v$ and $q$ as
\begin{align}
\label{eq:59}
u^{k+1} & \in \mathop{\argmin}_{v\in V} \  \mathscr{L}_{\delta}(v, p^k, \lambda^k),
&
 p^{k+1} & \in \mathop{\argmin}_{q\in H} \ \mathscr{L}_{\delta}(u^{k+1}, q, \lambda^k).
\end{align}
The ADMM algorithm for \eqref{eq:52}, described by \eqref{eq:55}, \eqref{eq:58}, and \eqref{eq:59}, can thus be summarized
\begin{subequations}
\label{eq:admm}
\begin{align}
u^{k+1} & \in \mathop{\argmin}_{v\in V} \  F(p^k) + E(v) +\langle\lambda^k, Bv-p^k\rangle_H + {{\delta} \over 2}\parallel Bv-p^k\parallel_H^2,
\label{eq:61}
\\
 p^{k+1} & \in \mathop{\argmin}_{q\in H} \  F(q) + E(u^{k+1}) +\langle\lambda^k, Bu^k-q\rangle_H + {{\delta} \over 2}\parallel Bu^k-q\parallel_H^2,
\label{eq:62}
\\
\lambda^{k+1} & = \lambda^k + \delta(Bu^{k+1} - p^{k+1}).
\label{eq:63}
\end{align}
\end{subequations}

\subsection{Proximal minimization}

When $V$ and $H$ are $\varmathbb{R}^n$, the sub-problems \eqref{eq:61} and \eqref{eq:62} can be solved by the proximal minimization algorithm that we now describe.
Let us consider the convex problem
\begin{equation}
\min_{x\in {\varmathbb{R}^n}} W(x),
\label{eq:64}
\end{equation}
where $W: \varmathbb{R}^n \to ]-\infty, +\infty]$ is a proper, lower semi-continuous convex function.
The Moreau--Yosida envelope of $W(x)$ is defined as
\begin{equation}
W_{\rho}(x) = \inf_{y \in \varmathbb{R}^n} \bigg\{W(y) + {1 \over {2\rho}}\|x-y\|^2\bigg\}.
\label{eq:65}
\end{equation}
As proved in \cite{MoJJ}, $W_{\rho}(x)$ is convex and differentiable, and has the same set of minimizers, and the same optimal value, as $W$.
This leads to the proximal minimization algorithm proposed by Martinet in \cite{MaB}. Namely, we solve \eqref{eq:64} by iterating
\begin{equation}
x^{k+1} = \mathop{\argmin}_{x\in \varmathbb{R}^n} \bigg\{W(x) + {1 \over {2\rho_{k+1}}}\|x-x^{k}\|^2\bigg\},
\label{eq:66}
\end{equation}
where the initial point $\ x^0 \in \varmathbb{R}^n$, $\{\rho_k\}_{k=1}^\infty$ is a sequence of positive numbers.
The convergence of this algorithm has been proved by Rockafellar in \cite{RoR1, RoR2}. For more discussion on proximal methods, see \cite{LeB,CoWa,Va2}.

\subsection{Preconditioned proximal minimization}

Following the ideas given in \cite{ZhBuOs}, we can improve the performance of the proximal minimization method \eqref{eq:66} by preconditioning. Specifically, we pick some positive definite symmetric matrix $Q$, and replace \eqref{eq:66} by
\begin{equation}
x^{k+1} = \mathop{\argmin}_{x\in \varmathbb{R}^n} \bigg\{W(x) + \rho_{k+1}^{-1}\  \big\| x - x^{k} \big\|_Q^2\bigg\}.
\label{eq:67}
\end{equation}
Here we define $\parallel x - x^k \parallel_Q \  := \sqrt{\langle Q(x - x^k), x - x^k\rangle}$.
Thus, picking ${Q_v^k}$ and ${Q_q^k}$ positive semi-definite, and incorporating the corresponding Moreau--Yosida regularization into \eqref{eq:admm}, we obtain the preconditioned ADMM (cf.~\cite{esser2010general,zhang2011unified})
\begin{subequations}%
\label{eq:admm-precond}
\begin{align}
u^{k+1} & = \mathop{\argmin}_{v\in \varmathbb{R}^n}~ F(p^k) + E(v) +\langle\lambda^k, Bv-p^k\rangle + {{\delta} \over 2}\big\| Bv-p^k\big\|^2 + \rho_{k+1}^{-1}\  \big\| v - u^{k} \big\|_{Q_v^k}^2,
\label{eq:68}
\\
 p^{k+1} & = \mathop{\argmin}_{q\in \varmathbb{R}^n}~ F(q) + E(u^{k+1}) +\langle\lambda^k, Bu^{k+1}-q\rangle + {{\delta} \over 2}\big\| Bu^{k+1}-q\big\|^2 + \rho_{k+1}^{-1}\  \big\| q - p^{k} \big\|_{Q_q^k}^2,
\label{eq:69}
\\
\lambda^{k+1} & = \lambda^k + \delta(Bu^{k+1} - p^{k+1}).
\label{eq:70}
\end{align}
\end{subequations}

\subsection{A computational method for the proposed model}

In order to cast the model \eqref{eq:9} in the preconditioned ADMM framework, let us first study the augmented Lagrangian for \eqref{eq:9}.
For $v = (u; a)^T$, we define
\begin{equation}
B(v) :=
\begin{pmatrix}
G(v),
\nabla u,
a
\end{pmatrix},   \label{eq:71}
\end{equation}
where $G$ and $a$ are as in \eqref{eq:50} and \eqref{eq:51}, respectively.
We represent the image as a vector $u \in \varmathbb{R}^{N^2}$, and write $\nabla u = (\nabla_1 u; \nabla_2 u) \in \varmathbb{R}^{2N^2}$.
with $E = 0$ and
\[
    F(B(v)) =
    \frac{1}{2}\sum_{j=1}^J\big\|P\mathcal{F}(G(v))_j - g_j\big\|_2^2 + \alpha_0R_u(u) + \alpha R_a(a),
\]
the problem \eqref{eq:9} has the form \eqref{eq:52} except for $B$ being nonlinear operator.
It follows from the above analysis and the preceding sections that an augmented Lagrangian naturally associated with the problem
\eqref{eq:9} is given by
\begin{equation}
\mathscr{L}_{\delta}(v, q, \mu) = F(q)  +(\mu, B(v)-q) + {{\delta} \over 2}\big\| B(v)-q\big\|^2. \label{eq:74}
\end{equation}
By analogue, following \cite{BeKnScVa}, we extend the preconditioned ADMM \eqref{eq:admm-precond} to non-linear $B$ as
\begin{subequations}
\begin{align}
u^{k+1} & = \mathop{\argmin}_{v}~F(p^k) + \langle\lambda^k, B(v)-p^k\rangle + {{\delta} \over 2}\big\| B(v)-p^k\big\|^2 + \rho_{k+1}^{-1}\  \big\| v - u^{k} \big\|_{Q_v^k}^2,
\label{eq:75}
\\
 p^{k+1} & = \mathop{\argmin}_{q}~F(q)+\langle\lambda^k, B(u^{k+1})-q\rangle + {{\delta} \over 2}\big\| B(u^{k+1})-q\big\|^2 + \rho_{k+1}^{-1}\  \big\| q - p^{k} \big\|_{Q_q^k}^2,
\label{eq:76}
\\
\lambda^{k+1} & = \lambda^k + \delta(B(u^{k+1}) - p^{k+1}).
\label{eq:77}
\end{align}
\end{subequations}
For our specific problem, the minimizations are over $v \in \varmathbb{R}^{N^2+J*{l_{\max}^{(\tilde{n})}}}$, and $q \in \varmathbb{R}^{J*N^2+2*N^2+J*{l_{\max}^{(\tilde{n})}}}$.

To make the proximal minimizations \eqref{eq:75} and \eqref{eq:76} easier, we linearise the operator $B$.
Let $\tilde{F_1}(v) = B(v) - p^k$ and $\tilde{F_2}(q) = B(u^{k+1}) - q$.
Since these functions are smooth,
\begin{align}
\label{eq:79}
\tilde{F_1}(v) & \approx \tilde{F_1}(u^k) + \tilde{F_1}'(u^k)(v - u^k),
\quad\text{and}
&
\tilde{F_2}(q) & \approx \tilde{F_2}(p^k) + \tilde{F_2}'(p^k)(v - p^k),
\end{align}
where $\tilde{F_1}'(u^k)$ and $\tilde{F_2}'(p^k)$ are the Fr\'echet derivative of $\tilde{F_1}$ at $u^k$ and $\tilde{F_2}$ at $p^k$.
For the sake of clarity, we set $J_v^k :=\tilde{F_1}'(u^k)$, $J_q^k := \tilde{F_2}'(p^k)$, $r_v^k := \tilde{F_1}'(u^k)u^k -
\tilde{F_1}(u^k)$ and $r_q^k := \tilde{F_2}'(p^k)p^k -
\tilde{F_2}(p^k)$.
Using \eqref{eq:79}, \eqref{eq:75}, and \eqref{eq:76} are replaced by
\begin{subequations}
\begin{align}
u^{k+1} & = \mathop{\argmin}_{v} \langle\lambda^k, J_q^kv\rangle + {{\delta} \over 2}\big\| J_q^kv-r_v^k\big\|^2 + \rho_{k+1}^{-1}\  \big\| v - u^{k} \big\|_{Q_v^k}^2,
\label{eq:80}
\\
 p^{k+1} & = \mathop{\argmin}_{q}~ F(q)+\langle\lambda^k, J_q^kv\rangle + {{\delta} \over 2}\big\| J_q^kv-r_v^k\big\|^2 + \rho_{k+1}^{-1}\  \big\| q - p^{k} \big\|_{Q_q^k}^2,
\label{eq:81}
\end{align}
\end{subequations}
where still $v \in \varmathbb{R}^{N^2+J*{l_{\max}^{(\tilde{n})}}}$, and $q \in \varmathbb{R}^{J*N^2+2*N^2+J*{l_{\max}^{(\tilde{n})}}}$.

In order to simplify the linearisation \eqref{eq:80} and \eqref{eq:81},
 we will seek $Q_v^k$ and $Q_q^k$ by $J_v^k$ and $J_q^k$.
For $\tau_v^k\delta < {1 \over \|J_v^k\|^2}$ and $\tau_q^k\delta < {1 \over \|J_q^k\|^2}$, we specifically let $Q_v^k := \tau_v^kI - \delta {J_v^k}^*{J_v^k}$ and $Q_q^k := \tau_q^kI - \delta {J_q^k}^*{J_q^k}$.
We also set $\bar{\lambda}^k := 2\lambda^k - \lambda^{k-1}$ and $\rho_{k+1} = 2$. It follows easily from \eqref{eq:80} and \eqref{eq:81}
\begin{align}
\label{eq:85}
u^{k+1} & = u^k - \tau_v^k{J_v^k}^*{\bar{\lambda}}^k ,
&
\\
p^{k+1} & = \bigg(I + \tau_q^k\partial F\bigg)^{-1}\bigg(p^k - \tau_q^k{J_q^k}^*\bigg(\lambda^k + \delta\bigg(B(u^{k+1})-p^k\bigg)\bigg)\bigg).
\label{eq:86}
\end{align}
Finally, \eqref{eq:85},  \eqref{eq:86} and \eqref{eq:77} yields \cref{alg:buildtree1} for the solution of \eqref{eq:9}.
Its convergence is studied in \cite{BeKnScVa} based on the results of \cite{Va1}.

\begin{algorithm}
\caption{Linearised preconditioned ADMM for \eqref{eq:9}}
\label{alg:buildtree1}
\begin{algorithmic}
\STATE{\textbf{Initialization} $u^0$, $p^0$, $\lambda^0$ and $\delta$.}
\STATE{\textbf{Set} $\bar{\lambda}^0 = \mu^0$.}
\WHILE{"stopping criterion is not satisfied"}
\STATE{$J_v^k = \tilde{F_1}'(u^k)$}
\STATE{Choose $\tau_v^k$ such that $\tau_v^k\delta < {1 \over \|J_v^k\|^2}$}
\STATE{Compute $u^{k+1}$ by \eqref{eq:85}}
\STATE{Compute $J_q^k$ by $J_q^k = \tilde{F_2}'(p^k)$}
\STATE{Choose $\tau_q^k$ such that $\tau_q^k\delta < {1 \over \|J_q^k\|^2}$}
\STATE{Compute $p^{k+1}$ by \eqref{eq:86}}
\STATE{Compute $\lambda^{k+1}$ by \eqref{eq:77}}
\STATE{Compute $\bar{\lambda}^{k+1} = 2\lambda^{k+1} - \lambda^k$}
\ENDWHILE
\STATE{\textbf{Return} $u^k$, $p^k$, $\lambda^k$ and $\bar{\lambda}^k$}
\end{algorithmic}
\end{algorithm}

\section{Numerical experiments}
\label{sec:num}

We now study the reconstruction performance of \eqref{eq:9} in comparison to the model \eqref{eq:variational} from \cite{BeKnScVa}.

\subsection{Technical details}

In the numerical simulations, we introduce regularization parameters $\alpha_j$ into the fidelity term $\frac{1}{2}\sum_{j=1}^J\big\|P\mathcal{F}(G(v))_j - g_j\big\|_2^2$ in the model \eqref{eq:9}. With this, the objective functional becomes
\[
    \begin{split}
    F(Bv) & =
    \frac{1}{2}\sum_{j=1}^J\alpha_j\big\|P\mathcal{F}(G(v))_j - g_j\big\|_2^2 + \alpha_0R_u(u) + \alpha R_a(a)
    \\
    &
    = \frac{1}{2}\sum_{j=1}^J\alpha_j\big\|P\mathcal{F}(G(v))_j - g_j\big\|_2^2 + \alpha_0TV_I(u) +
     \alpha \sum_{j=1}^J\sum_{l=1}^{l_{\max}^{(\tilde{n})}}\big|a_l^{(j)}\big|.
    \end{split}
\]
We decompose $F(Bv)$ into
$F_j(v) = \frac{1}{2}\alpha_j\big\|P\mathcal{F}(G(v))_j- g_j\big\|_2^2$ for $j = 1, 2, \ldots, J$, as well as $F_{J+1}(u) = \alpha_0TV_I(u)$, and $F_{J+2}(a) = \alpha \sum_{j=1}^J\sum_{l=1}^{l_{\max}^{(\tilde{n})}}\big|a_l^{(j)}\big|$. That is $F = \sum_{j=1}^JF_j + F_{J+1} + F_{J+2} = \sum_{j=1}^{J+2}F_j$.
We compute the corresponding resolvents explicitly as
\begin{align}
\bigg(I + \tau_q^k\partial F_j\bigg)^{-1}(x) &
 = \mathcal{F}^{-1}\Bigg({{\mathcal{F}x_j + {\alpha_j \tau_q^k}P^Tg_j} \over {1+\alpha_j\tau_q^k \diag(P^TP)}}\Bigg),
\ j = 1, 2, \ldots, J,
\label{eq:88}
\\
\bigg(I + \tau_q^k\partial F_{J+1}\bigg)^{-1}(x) &
 = {{x_{J+1}} \over {TV_I\big(x_{J+1}\big)}}\max\Bigg(TV_I\big(x_{J+1}\big)-{\alpha_0\tau_q^k}, 0\Bigg),
\label{eq:89}
\\
\bigg(I + \tau_q^k\partial F_{J+1}\bigg)^{-1}(x) &
= {(x_{J+2})_i \over {\big|(x_{J+2})_i\big|}}\max\bigg(|(x_{J+2})_i|-{\alpha\tau_q^k}, 0\bigg),
\ i = 1, 2, \ldots, J*l_{\max}^{(\tilde{n})}.
\label{eq:90}
\end{align}

\subsection{Experimental setup}
\label{sec:experimental}

Our numerical experiments are based on the synthetic brain phantom from \cite{AuBrEvCo,BeKiZh}, depicted in \cref{fig:a} and of dimension $190 \times 190$.
It contains several tissues, such as cerebrospinal fluid (CSF), gray matter (GM), white matter (WM) and cortical bone.
In the numerical simulations, we set the number of coils $J = 8$. For the generation of $k$-space measurement data $g_j$, $j = 1, 2, \ldots, J$, we use the approach of \cite{BeKnScVa}. We generate 8 coil sensitivity maps, based on a measurement of a water bottle with an 8-channel head coil array. These measurements are in \cref{fig:2}. We then multiply the brain phantom with each of these coil sensitivity maps separately, and convert the result to k-space data with the Fourier transform. Then we apply the $25\%$ subsampling mask shown in \cref{fig:b}. Finally, we add Gaussian noise with standard deviation $\tilde{\sigma}$ to the sub-sampled data.

We also demonstrate the robustness of the proposed approach in the p-MRI reconstruction by perturbing the coil sensitivity maps obtained from the water bottle. This is done by adding the 1st spherical basis function multiplied by factor $\gamma = 4$ to the water bottle measurements. The resulting maps are shown in \cref{fig:8}.

In numerical experiments, for the number of spherical basis functions ``levels'', we choose either $\tilde{n} = 2$ or $\tilde{n} = 5$. So the number of spherical basis functions is either $l_{\max}^{(\tilde{n})} = 9$ or $l_{\max}^{(\tilde{n})} = 36$. 
In the Cartesian coordinate system, we set $z_0 =0.5$, and $x, y$ in MATLAB are discretised by 
$${{2*step*(1:190)} \over 190} - 10,$$
where $step = 10$.
We plot in \cref{fig:4} the first 36 spherical basis functions corresponding to $\tilde{n} = 5$. For $\tilde n=2$ only the 9 first are used.

\begin{figure}[tbp]
    \centering
    \subcaptionbox{Brain phantom}{\label{fig:a}\includegraphics[width=0.3\textwidth]{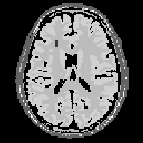}}
    \hfil%
    \subcaptionbox{$k$-space sampling mask}{\label{fig:b}\includegraphics[width=0.3\textwidth]{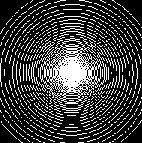}}
    \caption{ (a) shows the brain phantom described in \cref{sec:experimental}. (b) shows the spiral-shaped $25\%$ $k$-space sub-sampling mask.}
    \label{fig:1}
\end{figure}

\begin{figure}
  \centering
  {\includegraphics[width=0.15\textwidth]{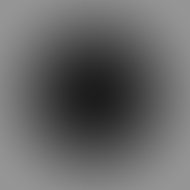}}
  {\includegraphics[width=0.15\textwidth]{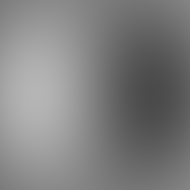}}
  {\includegraphics[width=0.15\textwidth]{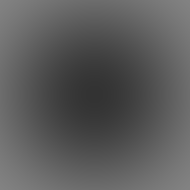}}
  {\includegraphics[width=0.15\textwidth]{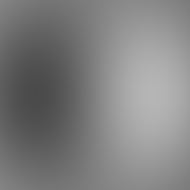}}
  {\includegraphics[width=0.15\textwidth]{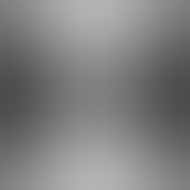}}
  {\includegraphics[width=0.15\textwidth]{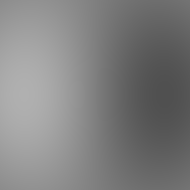}}\\
  {\includegraphics[width=0.15\textwidth]{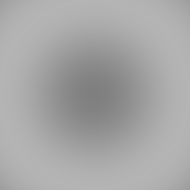}}
  {\includegraphics[width=0.15\textwidth]{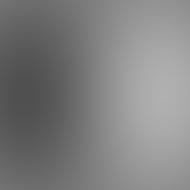}}
  {\includegraphics[width=0.15\textwidth]{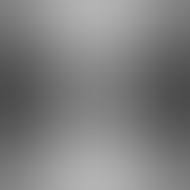}}
  {\includegraphics[width=0.15\textwidth]{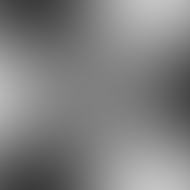}}
  {\includegraphics[width=0.15\textwidth]{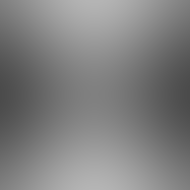}}
  {\includegraphics[width=0.15\textwidth]{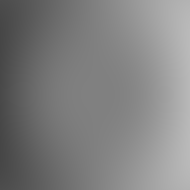}}\\
  {\includegraphics[width=0.15\textwidth]{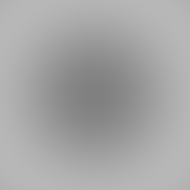}}
  {\includegraphics[width=0.15\textwidth]{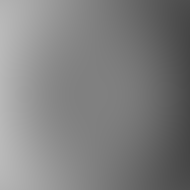}}
  {\includegraphics[width=0.15\textwidth]{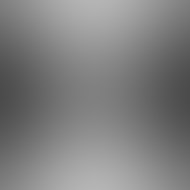}}
  {\includegraphics[width=0.15\textwidth]{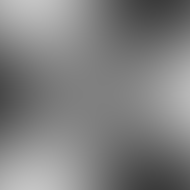}}
  {\includegraphics[width=0.15\textwidth]{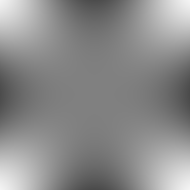}}
  {\includegraphics[width=0.15\textwidth]{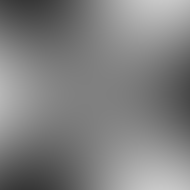}}\\
  {\includegraphics[width=0.15\textwidth]{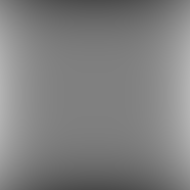}}
  {\includegraphics[width=0.15\textwidth]{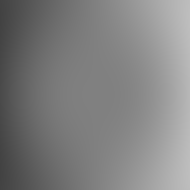}}
  {\includegraphics[width=0.15\textwidth]{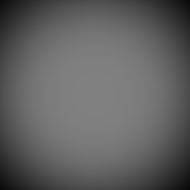}}
  {\includegraphics[width=0.15\textwidth]{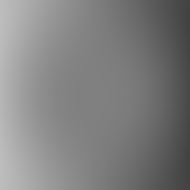}}
  {\includegraphics[width=0.15\textwidth]{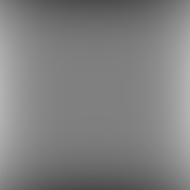}}
  {\includegraphics[width=0.15\textwidth]{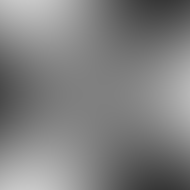}}\\
  {\includegraphics[width=0.15\textwidth]{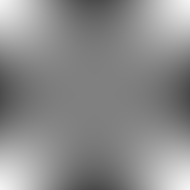}}
  {\includegraphics[width=0.15\textwidth]{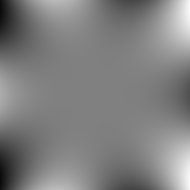}}
  {\includegraphics[width=0.15\textwidth]{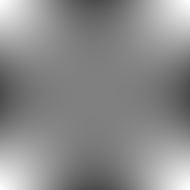}}
  {\includegraphics[width=0.15\textwidth]{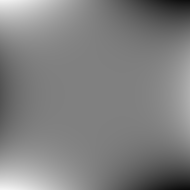}}
  {\includegraphics[width=0.15\textwidth]{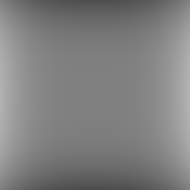}}
  {\includegraphics[width=0.15\textwidth]{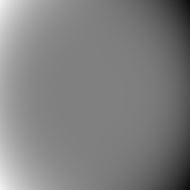}}\\
  {\includegraphics[width=0.15\textwidth]{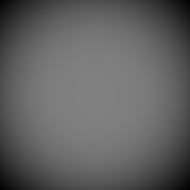}}
  {\includegraphics[width=0.15\textwidth]{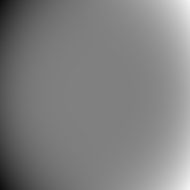}}
  {\includegraphics[width=0.15\textwidth]{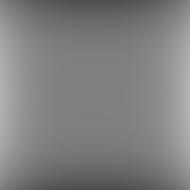}}
  {\includegraphics[width=0.15\textwidth]{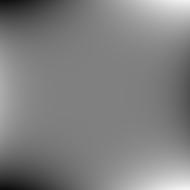}}
  {\includegraphics[width=0.15\textwidth]{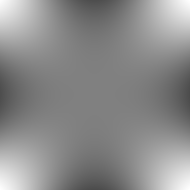}}
  {\includegraphics[width=0.15\textwidth]{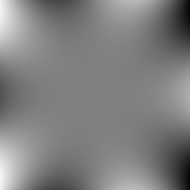}}

\caption{The 36 first spherical basis functions corresponding to $\tilde{n} = 5$. For $\tilde n=2$ only the 9 first are used.}
  \label{fig:4}
\end{figure}

\subsection{Quality measures and parameter selection}

All algorithms have been implemented in MATLAB, and the test hardware is an
Intel Core i7-6700 HQ CPU 2.60GHz with 8GB RAM.
We evaluate the results of the proposed approach in terms of the peak signal-to-noise (PSNR) that is available in the image processing toolbox in MATLAB and the Structural SIMilarity (SSIM) given in \cite{wang2004ssim}.
In the computation of the spherical basis function ${f_l^{+}}$, we use the Larmour frequency $\omega = 42.58$.
The conductivity $\sigma$ and the dielectric permittivity $\varepsilon$ are the optimal $(\sigma, \varepsilon)$ for the heterogeneous model in \cite{SbHoLaLuVa} with $\sigma = 0.6, \varepsilon = 50$.
The magnetic permeability for water is $\mu = 1.2566\times 10^{-6}$.

\subsection{Numerical reconstructions and comparison between \eqref{eq:9} and \eqref{eq:variational}}
We perform experiments with the noise level $\tilde{\sigma} = 0.05$.
We initialize the \cref{alg:buildtree1} with $u^0 = 0 \in \varmathbb{R}^{N^2}$, $a^0 = (1, \ldots, 1) \in {\varmathbb{R}^{Jl_{\max}^{(n)}}}$, and take as step lengths $\tau_v^k = 1/8$, $\tau_q^k = 23$, and $\delta = 1/24$ and the remaining $v^0$, $\lambda^0$, $\bar{\lambda}^0$ are all also initialized to zero in the two algorithms. For numerical reconstruction corresponding to \eqref{eq:variational}, we use the codes from \cite{BeKnScVa}, available from \cite{BeKnScVa-codes}.

We take as regularization parameters $\alpha_j = 0.4018$, ($j=1,\ldots,J$), $\alpha_0 = 0.0062$ and $\alpha =0.2149$.
We perform a fixed number of iterations of \cref{alg:buildtree1}.
With $\tilde{n} = 2$, i.e., $l_{\max} = 9$, stopping after 1000, 1200, and 1500 iterations, the reconstruction results for the model \eqref{eq:9} are shown in \cref{fig:5}. We also perform the numerical simulations with $\tilde{n} = 5$, and the same number of iterations 1000, 1200, and 1500. The reconstruction results for the model \eqref{eq:9} are shown in \cref{fig:6}, and for the model \eqref{eq:variational} in \cref{fig:11}. In \cref{tab:table6} we report the PNSR and SSIM \cite{wang2004ssim} values. From these results we can observe that the reconstruction quality of the model \eqref{eq:9} is much better than the model \eqref{eq:variational} when iterations are 1000, 1200, and 1500. The reconstruction results can be further improved due to the regularization parameters $\alpha$ not being optimally chosen; for a truly fair comparison of the potential of the two models distinct, parameter learning strategies should be used \cite{delosreyes2014learning,tuomov-tgvlearn}. What we can with reasonable confidence say based on our experiments here is that the non-linear ADMM converges faster for the model \eqref{eq:9} than for \eqref{eq:variational}. This is important in practical applications.

We also report the absolute values of the coefficients $a_l^{(j)}$ in \cref{tab:table7} when the stopping number of iterations is 1500 for $\tilde{n} = 2$.  Similarly, \cref{tab:table8} shows the absolute values for the resulted coefficients $a_l^{(j)}$ for $\tilde{n} = 5$ when we stop at 1500 iterations.
While for $\tilde{n}=2$ the last rows of the coefficient pyramid for each coil still have high coefficient values, for $\tilde{n}=5$ the coefficients on the last row have decayed to below 1\% of the main coefficient on the first row; often 0.1\% or less. This supports our starting intuition that a sparse approximation of the coil sensitivities with relatively few coefficients is sufficient for a high-quality reconstruction.

Using the discovered coefficients $a_l^{(j)}$, and the known spherical basis functions, we can reconstruct the approximation of the coil sensitivities $c_j$. These are in \cref{fig:7} and \cref{fig:75} for $\tilde{n} = 2$ and $\tilde{n} = 5$ with 1500 iterations. In order to do the further comparison between \eqref{eq:9} and \eqref{eq:variational}, we also give the approximation of $c_j$ with 1500 iterations for \eqref{eq:variational} in \cref{fig:79}. The PSNR and SSIM values for reconstruction of coils using \eqref{eq:9} and \eqref{eq:variational} for 1500 iterations are reported in \cref{tab:table9}. Visually, the coil sensitivities constructed with our model \eqref{eq:variational} are significantly smoother than those constructed with the model \eqref{eq:9}, and indeed appear to very well approximate the ``true'' coil sensitivities in \cref{fig:2}.

To test robustness, we show in \cref{fig:9,fig:12} for $\tilde n=2$ and $\tilde n=5$, respectively, the reconstructions results for the alternative coil sensitivity maps in \cref{fig:8}. The number of iterations is 1500. The PSNR and SSIM values are reported in \cref{tab:table10}. Comparing to \cref{fig:5,fig:6,tab:table6}, we can see that the results remain stable under this perturbation of coil sensitivities, being virtually identical. By contrast, the reconstructed coil sensitivities in \cref{fig:10,fig:105} have changed, corresponding to the change in true coil sensitivities.

\begin{figure}[tbp]
    \centering
    \subcaptionbox{1200 iterations}{\label{fig:5a}\includegraphics[width=0.3\textwidth]{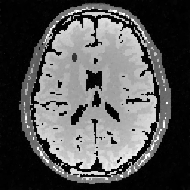}}
    \ %
    \subcaptionbox{1500 iterations}{\label{fig:5b}\includegraphics[width=0.3\textwidth]{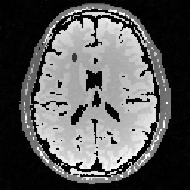}}
    \ %
    \subcaptionbox{1800 iterations}{\label{fig:5c}\includegraphics[width=0.3\textwidth]{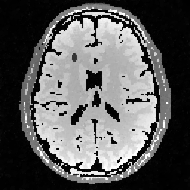}}
    \caption { Reconstructed brain slice using \eqref{eq:9} and $\tilde{n} = 2$.
    }
    \label{fig:5}
\end{figure}

\begin{figure}[tbp]
    \centering
    \subcaptionbox{1200 iterations}{\label{fig:6a}\includegraphics[width=0.3\textwidth]{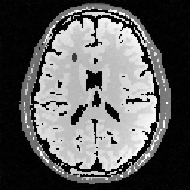}}
    \ %
    \subcaptionbox{1500 iterations}{\label{fig:6b}\includegraphics[width=0.3\textwidth]{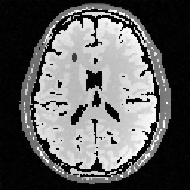}}
    \ %
    \subcaptionbox{1800 iterations}{\label{fig:6c}\includegraphics[width=0.3\textwidth]{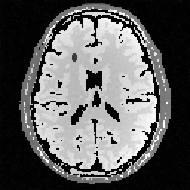}}
    \caption { Reconstructed brain slice using \eqref{eq:9} and $\tilde{n} = 5$. 
    }
    \label{fig:6}
\end{figure}

\begin{figure}[tbp]
    \centering
    \subcaptionbox{1200 iterations}{\label{fig:11a}\includegraphics[width=0.3\textwidth]{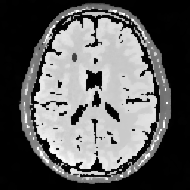}}
    \ %
    \subcaptionbox{1500 iterations}{\label{fig:11b}\includegraphics[width=0.3\textwidth]{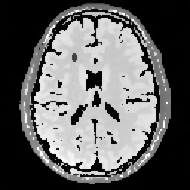}}
    \ %
    \subcaptionbox{1800 iterations}{\label{fig:11c}\includegraphics[width=0.3\textwidth]{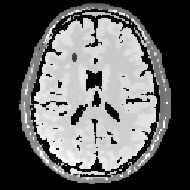}}
    \caption { Reconstructed brain slice using \eqref{eq:variational}.
    }
    \label{fig:11}
\end{figure}

\begin{table}[tbp]
  \caption{Reconstruction quality comparison between \eqref{eq:9} with $\tilde{n} = 2, 5$ and \eqref{eq:variational}.}
  \label{tab:table6}
  \centering
  \begin{tabular}{|c|c|c|c|} \hline
  Method &  stopping itr. k  & PSNR(dB) & SSIM  \\ \hline
   &  1200 & 25.2750 & 0.9996 \\
  using \eqref{eq:9} with $\tilde{n} = 2$   & 1500 & 25.6878 & 0.9997 \\
  & 1800 &  25.1069 & 0.9996 \\  \hline
  &  1200 & 26.0725 & 0.9997 \\
  using \eqref{eq:9} with $\tilde{n} = 5$   & 1500 & 25.5883 & 0.9997\\
  & 1800 &  25.8730 & 0.9997\\  \hline
  &  1200 & 24.7173 & 0.9996\\
  using \eqref{eq:variational}   & 1500 & 24.1524 & 0.9996\\
  & 1800 &  23.6702 & 0.9995\\  \hline
  \end{tabular}
\end{table}

\begin{table}[tbp]
  \caption{The absolute values of coefficients of $f_l^{+}$ with $\tilde{n} = 2$ for 1500 iterations. j is the ordinal number of coils.}
  \label{tab:table7}
  \centering\small
  \begin{tabular}{|c|c|} \hline
j  & $f_1^{+}$ \\
   & $f_2^{+}, f_3^{+}, f_4^{+}$ \\
   & $f_5^{+}, f_6^{+}, f_7^{+}, f_8^{+}, f_9^{+}$ \\ \hline
   & 5.4801\\
  1 & 0.9382, 0.0092, 4.7704 \\
   & 1.0213, 0.0075, 5.3347, 0.0034, 3.8693 \\ \hline
   & 4.8840   \\
  2 & 0.3632, 0.0059, 4.7055 \\
   & 1.2463, 0.0085, 4.6682, 0.0083, 4.3803 \\ \hline
   & 3.4124  \\
  3 & 0.2777, 0.0058, 4.3643 \\
   & 1.7627, 0.0072, 2.2884, 0.0066, 4.1581 \\ \hline
   & 2.6890 \\
  4 & 0.1031, 0.0094, 4.2809 \\
   & 1.4676, 0.0065, 2.5603, 0.0058, 3.7494 \\ \hline
  \end{tabular}\ %
  \begin{tabular}{|c|c|} \hline
j  & $f_1^{+}$ \\
   & $f_2^{+}, f_3^{+}, f_4^{+}$ \\
   & $f_5^{+}, f_6^{+}, f_7^{+}, f_8^{+}, f_9^{+}$ \\ \hline
    & 3.4487 \\
  5 & 0.2859, 0.0096, 3.6825 \\
   & 1.1578, 0.0075, 3.1660, 0.0059, 3.2907 \\ \hline
   & 4.3978 \\
  6 & 0.1909, 0.0085, 3.4326  \\
   & 0.6340, 0.0060, 3.7131, 0.0048, 3.2575 \\ \hline
   & 3.7714   \\
  7 & 0.2415, 0.0016, 3.0790  \\
   & 1.4061, 0.0070, 4.1169, 0.0045, 3.8595 \\ \hline
   & 3.3760  \\
  8 & 0.2552, 0.0065, 3.2423  \\
   & 1.6636, 0.0093, 2.7457, 0.0075, 3.7089  \\ \hline
\end{tabular}
\end{table}

\makeatletter
\newcommand*{\centerfloat}{%
  \parindent \z@
  \leftskip \z@ \@plus 1fil \@minus \textwidth
  \rightskip\leftskip
  \parfillskip \z@skip}
\makeatother

\begin{sidewaystable}[tbp]
  \caption{The absolute values of coefficients of $f_l^{+}$ with $\tilde{n} = 5$ for 1500 iterations. j is the ordinal number of coils.}
  \label{tab:table8}
  \centering\small
  \begin{tabular}{|c|c|} \hline
   &  $f_1^{+}$ \\
   & $f_2^{+}, f_3^{+}, f_4^{+}$ \\
j  & $f_5^{+}, f_6^{+}, f_7^{+}, f_8^{+}, f_9^{+}$\\
   & $f_{10}^{+}, f_{11}^{+}, f_{12}^{+}, f_{13}^{+}, f_{14}^{+}, f_{15}^{+}, f_{16}^{+}$\\
   & $f_{17}^{+}, f_{18}^{+}, f_{19}^{+}, f_{20}^{+}, f_{21}^{+}, f_{22}^{+}, f_{23}^{+}, f_{24}^{+}, f_{25}^{+}$\\
   & $f_{26}^{+}, f_{27}^{+}, f_{28}^{+}, f_{29}^{+}, f_{30}^{+}, f_{31}^{+}, f_{32}^{+}, f_{33}^{+}, f_{34}^{+}, f_{35}^{+}, f_{36}^{+}$\\  \hline
   & \scriptsize{4.9921} \\
   & \scriptsize{.6829, .0438, 4.3006}\\
1  & \scriptsize{.8681,.0165, 5.4029,.0174, 3.9672}\\
   & \scriptsize{.9972, .0101, 1.9630, .0107, 1.5569, .0051, 5.2987}\\
   & \scriptsize{.0767, .0075, .0324, .0087, .0344, .0344, .0089, .0135, 5.3953}\\
   & \scriptsize{.0009, .0085, .0078, .0054, .0242, .0053, .0490, .0074, .0038, .0033, .0069}\\ \hline
   & \scriptsize{4.4692} \\
   & \scriptsize{.1920, .0465, 4.2089}\\
2  & \scriptsize{1.2692, .0071, 4.6202, .0165, 4.1568}\\
   & \scriptsize{.9978, .0050, .0230, .0072, 1.5304, .0038, 4.2382}\\
   & \scriptsize{.6541, .0098, .0122, .0070, .0428, .0284, .0251, .0071, 5.4117}\\
   & \scriptsize{.0031, .0075, .0048, .0064, .0066, .0054, .0288, .0059, .0052, .0059, .0040}\\ \hline
   & \scriptsize{3.0346} \\
   & \scriptsize{.3015, .0392, 3.7538}\\
3  & \scriptsize{1.6553, .0018, 2.9996, .0041, 4.1566}\\
   & \scriptsize{.9641, .0057, .3390, .0085, 2.2158, .0087, 3.9728}\\
   & \scriptsize{.0021, .0080, .0025, .0139, .0255, .0333, .0106, .0030, 4.4230}\\
   & \scriptsize{.0013, .0075, .0064, .0072, .0030, .0075, .0261, .0075, .0059, .0073, .0048}\\ \hline
   & \scriptsize{2.4995} \\
   & \scriptsize{.1197, .0360, 3.3852}\\
4  & \scriptsize{1.4643, .0034, 2.4634, .0056, 3.4846}\\
   & \scriptsize{1.0066, .0118, .1241, .0061, 1.5655, .0049, 4.3484}\\
   & \scriptsize{.9035, .0058, .0011, .0142, .0158, .0383, .0126, .0048, 5.1540}\\
   & \scriptsize{.0043, .0079, .0081, .0076, .0034, .0074, .0268, .0071, .0078, .0080, .0095}\\ \hline
\end{tabular}\ %
  \begin{tabular}{|c|c|} \hline
   &  $f_1^{+}$ \\
   & $f_2^{+}, f_3^{+}, f_4^{+}$ \\
j  & $f_5^{+}, f_6^{+}, f_7^{+}, f_8^{+}, f_9^{+}$\\
   & $f_{10}^{+}, f_{11}^{+}, f_{12}^{+}, f_{13}^{+}, f_{14}^{+}, f_{15}^{+}, f_{16}^{+}$\\
   & $f_{17}^{+}, f_{18}^{+}, f_{19}^{+}, f_{20}^{+}, f_{21}^{+}, f_{22}^{+}, f_{23}^{+}, f_{24}^{+}, f_{25}^{+}$\\
   & $f_{26}^{+}, f_{27}^{+}, f_{28}^{+}, f_{29}^{+}, f_{30}^{+}, f_{31}^{+}, f_{32}^{+}, f_{33}^{+}, f_{34}^{+}, f_{35}^{+}, f_{36}^{+}$\\  \hline
   & \scriptsize{3.1244} \\
   & \scriptsize{.2371, .0362, 3.0439}\\
5  & \scriptsize{.9841, .0108, 3.2107, .0073, 2.8668}\\
   & \scriptsize{.8269, .0041, .0164, .0065, .5001, .0153, 4.3171}\\
   & \scriptsize{.0811, .0057, .0055, .0057, .0239, .0266, .0237, .0025, 4.6823}\\
   & \scriptsize{.0025, .0061, .0072, .0077, .0125, .0088, .0353, .0091, .0052, .0111, .0058}\\ \hline
   & \scriptsize{3.9430} \\
   & \scriptsize{.3079, .0451, 3.1557}\\
6  & \scriptsize{.6159, .0090, 4.0539, .0148, 2.8496}\\
   & \scriptsize{1.3257, .0072, .0080, .0128, 1.5127, .0052, 3.1355}\\
   & \scriptsize{.3157, .0061, .0212, .0080, .0327, .0422, .0358, .0122, 4.5738}\\
   & \scriptsize{.0047, .0071, .0085, .0083, .0087, .0076, .0341, .0062, .0099, .0074, .0064}\\ \hline
   & \scriptsize{3.4755} \\
   & \scriptsize{.2479, .0338, 2.9170}\\
7  & \scriptsize{1.3750, .0061, 3.7187, .0110, 3.4468}\\
   & \scriptsize{.8354, .0115, .1450, .0138, 1.4890, .0046, 3.9267}\\
   & \scriptsize{.0038, .0074, .0169, .0107, .0299, .0389, .0244, .0157, 4.2940}\\
   & \scriptsize{.0028, .0078, .0054, .0075, .0021, .0071, .0275, .0055, .0043, .0068, .0051}\\ \hline
   & \scriptsize{2.9650} \\
   & \scriptsize{.2303, .0268, 2.9961}\\
8  & \scriptsize{1.5204, .0118, 3.1554, .0206, 3.4432}\\
   & \scriptsize{1.3120, .0106, .1190, .0056, 1.3044, .0064, 4.4748}\\
   & \scriptsize{.6673, .0071, .0079, .0051, .0222, .0196, .0165, .0123, 4.9405}\\
   & \scriptsize{.0031, .0080, .0065, .0070, .0114, .0076, .0177, .0087, .0153, .0102, .0076}\\ \hline
\end{tabular}
\end{sidewaystable}

\begin{figure}
  \centering
  {\includegraphics[width=0.20\textwidth]{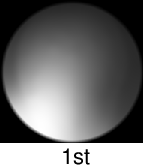}}
  {\includegraphics[width=0.20\textwidth]{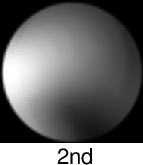}}
  {\includegraphics[width=0.20\textwidth]{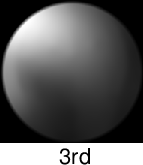}}
  {\includegraphics[width=0.20\textwidth]{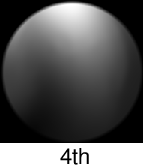}}\\
  {\includegraphics[width=0.20\textwidth]{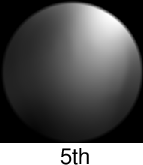}}
  {\includegraphics[width=0.20\textwidth]{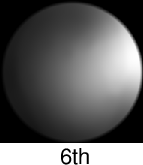}}
  {\includegraphics[width=0.20\textwidth]{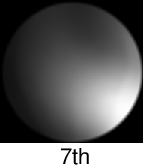}}
  {\includegraphics[width=0.20\textwidth]{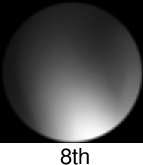}}
  \caption{ Coil sensitivities generated by the measurements of a water bottle with 8-channel head coil array.}
  \label{fig:2}
\end{figure}

\begin{figure}
  \centering
  {\includegraphics[width=0.20\textwidth]{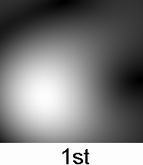}}
  {\includegraphics[width=0.20\textwidth]{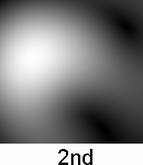}}
  {\includegraphics[width=0.20\textwidth]{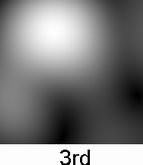}}
  {\includegraphics[width=0.20\textwidth]{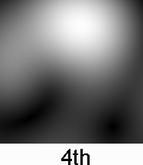}}\\
  {\includegraphics[width=0.20\textwidth]{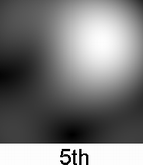}}
  {\includegraphics[width=0.20\textwidth]{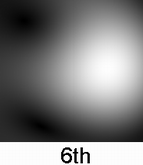}}
  {\includegraphics[width=0.20\textwidth]{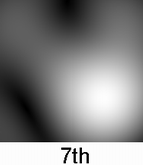}}
  {\includegraphics[width=0.20\textwidth]{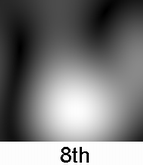}}
  \caption{ Reconstructed coil sensitivities for the 8 different coils when the number of iterations is 1500 for $\tilde{n} = 2$.}
  \label{fig:7}
\end{figure}

\begin{figure}
  \centering
  {\includegraphics[width=0.20\textwidth]{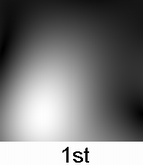}}
  {\includegraphics[width=0.20\textwidth]{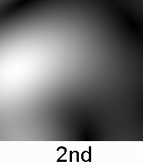}}
  {\includegraphics[width=0.20\textwidth]{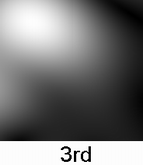}}
  {\includegraphics[width=0.20\textwidth]{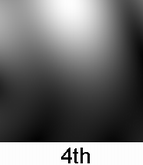}}\\
  {\includegraphics[width=0.20\textwidth]{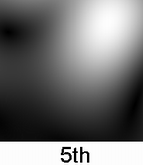}}
  {\includegraphics[width=0.20\textwidth]{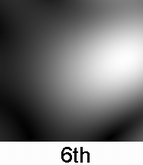}}
  {\includegraphics[width=0.20\textwidth]{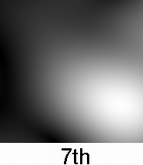}}
  {\includegraphics[width=0.20\textwidth]{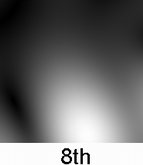}}
  \caption{ Reconstructed coil sensitivities for the 8 different coils when the number of iterations is 1500 for $\tilde{n} = 5$.}
  \label{fig:75}
\end{figure}

\begin{figure}
  \centering
  {\includegraphics[width=0.20\textwidth]{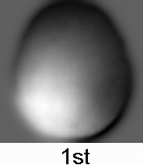}}
  {\includegraphics[width=0.20\textwidth]{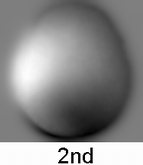}}
  {\includegraphics[width=0.20\textwidth]{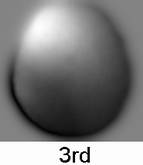}}
  {\includegraphics[width=0.20\textwidth]{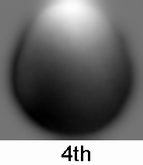}}\\
  {\includegraphics[width=0.20\textwidth]{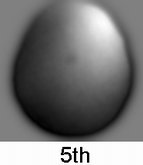}}
  {\includegraphics[width=0.20\textwidth]{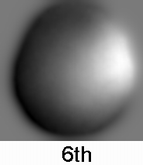}}
  {\includegraphics[width=0.20\textwidth]{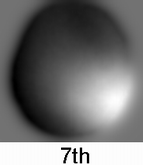}}
  {\includegraphics[width=0.20\textwidth]{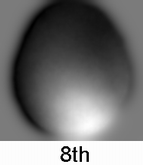}}
  \caption{ Reconstructed coil sensitivities for the 8 different coils when the number of iterations is 1500 for \eqref{eq:variational}.}
  \label{fig:79}
\end{figure}

\begin{figure}
    \centering
    \subcaptionbox{1200 iterations}{\label{fig:9a}\includegraphics[width=0.3\textwidth]{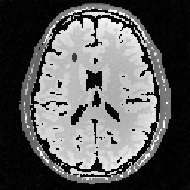}}
    \ %
    \subcaptionbox{1500 iterations}{\label{fig:9b}\includegraphics[width=0.3\textwidth]{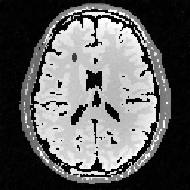}}
    \ %
    \subcaptionbox{1800 iterations}{\label{fig:9c}\includegraphics[width=0.3\textwidth]{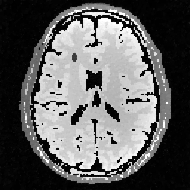}}
    \caption{ Reconstruction when data is generated with the alternative coil sensitivities from \cref{fig:8} with $\tilde{n} = 2$.}
    \label{fig:9}
\end{figure}

\begin{figure}
    \centering
    \subcaptionbox{1200 iterations}{\label{fig:12a}\includegraphics[width=0.3\textwidth]{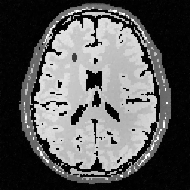}}
    \ %
    \subcaptionbox{1500 iterations}{\label{fig:12b}\includegraphics[width=0.3\textwidth]{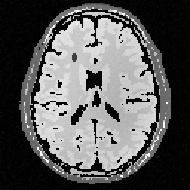}}
    \ %
    \subcaptionbox{1800 iterations}{\label{fig:12c}\includegraphics[width=0.3\textwidth]{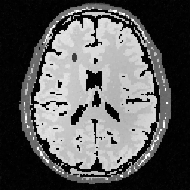}}
    \caption{ Reconstruction when data is generated with the alternative coil sensitivities from \cref{fig:8} with $\tilde{n} = 5$.}
  \label{fig:12}
\end{figure}

\begin{table}[tbp]
  \caption{Reconstruction quality comparison between \eqref{eq:9} with $\tilde{n} = 2, 5$ and \eqref{eq:variational} for 1500 iterations.}
  \label{tab:table9}
  \centering
  \begin{tabular}{|c|c|c|c|} \hline
  Method &  coil No.  & PSNR(dB) & SSIM  \\ \hline
  &  coil 1 & 8.5038 & 0.9864 \\
  &  coil 2 & 8.9329 & 0.9878 \\
  &  coil 3 & 8.9895 & 0.9873 \\
  &  coil 4 & 10.7526 & 0.9918 \\
  using \eqref{eq:9} with $\tilde{n} = 2$   &  coil 5 & 10.2633 & 0.9900 \\
  &  coil 6 & 10.6595 & 0.9917 \\
  &  coil 7 & 9.0644 & 0.9869 \\
  &  coil 8 &  11.2573 & 0.9929 \\  \hline
  &  coil 1 & 7.2807 & 0.9849 \\
  &  coil 2 & 8.2247 & 0.9878 \\
  &  coil 3 & 7.4537 & 0.9850 \\
  &  coil 4 & 9.3979 & 0.9903 \\
  using \eqref{eq:9} with $\tilde{n} = 5$   &  coil 5 & 9.1837 & 0.9891 \\
  &  coil 6 & 10.3233 & 0.9919 \\
  &  coil 7 & 8.5967 & 0.9876 \\
  &  coil 8 &  10.1097 & 0.9920 \\  \hline
  &  coil 1 & 5.9826 & 0.9759 \\
  &  coil 2 & 5.9328 & 0.9754 \\
  &  coil 3 & 5.9124 & 0.9748 \\
  &  coil 4 & 6.3114 & 0.9779 \\
  using \eqref{eq:variational}   &  coil 5 & 6.1844 & 0.9762 \\
  &  coil 6 & 6.1754 & 0.9765 \\
  &  coil 7 & 6.3423 & 0.9772 \\
  &  coil 8 &  6.4319 & 0.9774 \\  \hline
  \end{tabular}
\end{table}

\begin{table}[tbp]
  \caption{The PSNR and SSIM values for reconstruction with alternative coil sensitivities using \eqref{eq:9} with $\tilde{n} = 2$ and $\tilde{n} = 5$}
  \label{tab:table10}
  \centering
  \begin{tabular}{|c|c|c|c|} \hline
  Method &  stopping itr. k  & PSNR(dB) & SSIM  \\ \hline
   &  1200 & 27.0202 & 0.9998\\
  using \eqref{eq:9} with $\tilde{n} = 2$   & 1500 & 24.2239 & 0.9995 \\
  & 1800 &  24.9633 & 0.9996\\  \hline
  &  1200 & 30.8382 & 0.9999 \\
  using \eqref{eq:9} with $\tilde{n} = 5$   & 1500 & 30.7620 & 0.9999\\
  & 1800 &  30.6448 & 0.9999\\  \hline
\end{tabular}
\end{table}

\begin{figure}
  \centering
  {\includegraphics[width=0.20\textwidth]{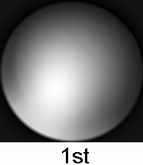}}
  {\includegraphics[width=0.20\textwidth]{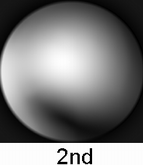}}
  {\includegraphics[width=0.20\textwidth]{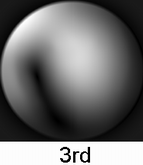}}
  {\includegraphics[width=0.20\textwidth]{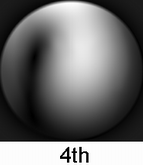}}\\
  {\includegraphics[width=0.20\textwidth]{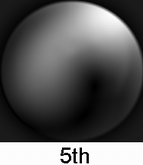}}
  {\includegraphics[width=0.20\textwidth]{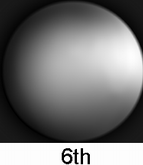}}
  {\includegraphics[width=0.20\textwidth]{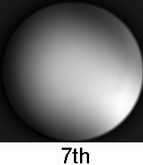}}
  {\includegraphics[width=0.20\textwidth]{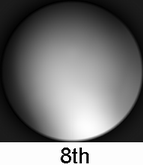}}
  \caption{ Alternative Perturbed coil sensitivity maps. }
  \label{fig:8}
\end{figure}

\begin{figure}
  \centering
  {\includegraphics[width=0.20\textwidth]{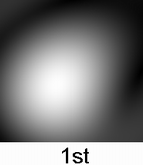}}
  {\includegraphics[width=0.20\textwidth]{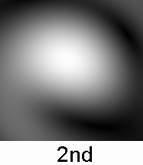}}
  {\includegraphics[width=0.20\textwidth]{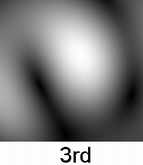}}
  {\includegraphics[width=0.20\textwidth]{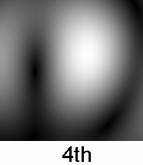}}\\
  {\includegraphics[width=0.20\textwidth]{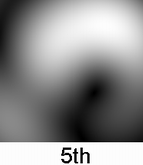}}
  {\includegraphics[width=0.20\textwidth]{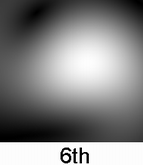}}
  {\includegraphics[width=0.20\textwidth]{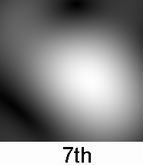}}
  {\includegraphics[width=0.20\textwidth]{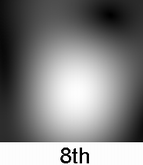}}
  \caption{ Reconstructed coil sensitivities when data is generated with the alternative coil sensitivities from \cref{fig:8}. The number of iterations is 1500 and $\tilde{n} = 2$.}
  \label{fig:10}
\end{figure}

\begin{figure}
  \centering
  {\includegraphics[width=0.20\textwidth]{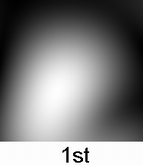}}
  {\includegraphics[width=0.20\textwidth]{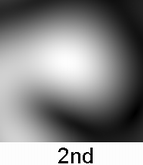}}
  {\includegraphics[width=0.20\textwidth]{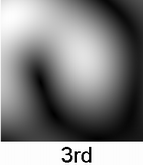}}
  {\includegraphics[width=0.20\textwidth]{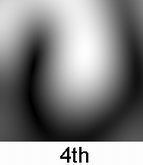}}\\
  {\includegraphics[width=0.20\textwidth]{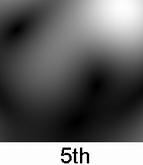}}
  {\includegraphics[width=0.20\textwidth]{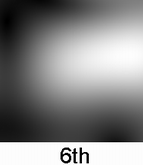}}
  {\includegraphics[width=0.20\textwidth]{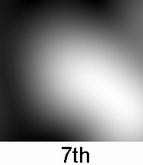}}
  {\includegraphics[width=0.20\textwidth]{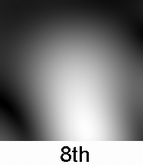}}
  \caption{ Reconstructed coil sensitivities when data is generated with the alternative coil sensitivities from \cref{fig:8}. The number of iterations is 1500 and $\tilde{n} = 5$.}
  \label{fig:105}
\end{figure}

\section{Conclusions}
\label{sec:conclusions}

In this paper, we have established a new model for parallel MRI reconstruction based on sparse regularization of coil sensitivities in spherical basis function bases. We have developed efficient recurrence formulas for the computation of these functions. We have then applied the non-linear ADMM from \cite{BeKnScVa} to numerically solve our model \eqref{eq:9}. By numerical reconstructions and comparison between \eqref{eq:9} and \eqref{eq:variational}, we think that the reconstruction quality for proposed model \eqref{eq:9} is better than the model \eqref{eq:variational}. In additional, the reconstruction for our model \eqref{eq:9} for the alternative coils sensitivity maps is very robust. That has an important significance in practical applications. In the future, we will study the optimal choice among the regularization parameters $\alpha_j$, $\alpha_0$, and $\alpha$ to improve reconstruction quality furthermore via parameter learning strategies in \cite{delosreyes2014learning,tuomov-tgvlearn}.

\section*{Acknowledgments}

Y.~Zhu was supported by the National Natural Science Foundation of China No.~11571325, Science Research Project of CUC No.~3132016XNL1612.
Towards the end of this research, T.~Valkonen has been supported by the EPSRC First Grant EP/P021298/1, ``PARTIAL Analysis of Relations in Tasks of Inversion for Algorithmic Leverage''.

We would like to thank Florian Knoll for the water bottle measurement, and Martin Benning for making his codes \cite{BeKnScVa-codes} available.

\section*{\texorpdfstring{\normalsize}{}A data statement for the EPSRC}

{\color{red}All data and source codes will be publicly deposited when the final accepted version of the manuscript is submitted.}

\printbibliography

\end{document}